\documentclass[11pt]{article}
\pdfoutput=1
\usepackage{enumerate}
\usepackage{pdfsync}
\usepackage[OT1]{fontenc}

\usepackage[usenames]{color}
\usepackage{smile}
\usepackage[colorlinks,
            linkcolor=red,
            anchorcolor=blue,
            citecolor=blue
            ]{hyperref}
\usepackage{fullpage}
\usepackage{hyperref}
\usepackage[protrusion=true,expansion=true]{microtype}
\usepackage{pbox}
\usepackage{setspace}
\usepackage{tabularx}
\usepackage{float}
\usepackage{wrapfig,lipsum}
\usepackage{enumitem}

\makeatletter
\newcommand*{\rom}[1]{\expandafter\@slowromancap\romannumeral #1@}
\makeatother

\def \EE {\mathbb{E}}
\def \RR {\mathbb{R}}

\def \xb {\mathbf{x}}
\def \yb {\mathbf{y}}

\def \ab {\mathbf{a}}
\def \hb {\mathbf{h}}

\def \bbb {\mathbf{b}}
\def \gb {\mathbf{g}}
\def \vb {\mathbf{v}}

\def \Vb {\mathbf{V}}
\def \Ib {\mathbf{I}}
\newcommand{\la}{\langle}
\newcommand{\ra}{\rangle}
\def \dF {\nabla F}
\def \HF {\nabla^2 F}
\def \df {\nabla f}
\def \Hf {\nabla^2 f}
\def \nameheavy {\text{SRVRC}}
\def \namefree {\nameheavy_\text{free}}
\def \namesub {\text{Cubic-Subsolver}}
\def \namefinal {\text{Cubic-Finalsolver}}
\def \mineig {\lambda_{\text{min}}}
\def \mod {\text{mod}}
\def \upp {M}

\begin{document}

\title{{\huge Stochastic Recursive Variance-Reduced Cubic Regularization Methods}}
\author
{
	Dongruo Zhou\thanks{Department of Computer Science, University of California, Los Angeles, CA 90095, USA; e-mail: {\tt drzhou@cs.ucla.edu}} 
	~~~and~~~
	Quanquan Gu\thanks{Department of Computer Science, University of California, Los Angeles, CA 90095, USA; e-mail: {\tt qgu@cs.ucla.edu}}
}
\date{}
\maketitle

\begin{abstract}
Stochastic Variance-Reduced Cubic regularization (SVRC) algorithms have received increasing attention due to its improved gradient/Hessian complexities (i.e., number of queries to stochastic gradient/Hessian oracles)  
to find local minima 
for nonconvex finite-sum optimization.
However, it is unclear whether existing SVRC algorithms can be further improved. Moreover, the semi-stochastic Hessian estimator adopted in existing SVRC algorithms prevents the use of Hessian-vector product-based fast cubic subproblem solvers, which makes SVRC algorithms computationally intractable for high-dimensional problems. 
In this paper, we first present a Stochastic Recursive Variance-Reduced Cubic regularization method (SRVRC) using a recursively updated semi-stochastic gradient and Hessian estimators. It enjoys improved gradient and Hessian complexities to find an $(\epsilon, \sqrt{\epsilon})$-approximate local minimum, and outperforms the state-of-the-art SVRC algorithms. 
Built upon SRVRC, we further propose a Hessian-free SRVRC algorithm, namely SRVRC$_{\text{free}}$, which only needs $\tilde O(n\epsilon^{-2} \land \epsilon^{-3})$ stochastic gradient and Hessian-vector product computations, where $n$ is the number of component functions in the finite-sum objective and $\epsilon$ is the optimization precision. This outperforms the best-known result $\tilde O(\epsilon^{-3.5})$ achieved by stochastic cubic regularization algorithm proposed in \cite{tripuraneni2018stochastic}. 
\end{abstract}


\section{Introduction}

Many machine learning problems can be formulated as empirical risk minimization, which is in the form of finite-sum optimization as follows:
\begin{align}
    \textstyle{\min_{\xb \in \RR^d}} F(\xb) := \frac{1}{n}\textstyle{\sum_{i=1}^n} f_i(\xb), \label{def:problem}
\end{align}
where each $f_i:\RR^d\rightarrow \RR$ can be a convex or nonconvex function. In this paper, we are particularly interested in nonconvex finite-sum optimization, where each $f_i$ is nonconvex.  This is often the case for deep learning \citep{lecun2015deep}.  In principle, it is hard to find the global minimum of \eqref{def:problem} because of the NP-hardness of the problem \citep{hillar2013most}, thus it is reasonable to resort to finding local minima (a.k.a., second-order stationary points). It has been shown that local minima can be the global minima in certain machine learning problems, such as low-rank matrix factorization \citep{ge2016matrix,bhojanapalli2016global,zhang2018primal} and training deep linear neural networks \citep{kawaguchi2016deep,hardt2016identity}. Therefore, developing algorithms to find local minima is important both in theory and in practice. More specifically, we define an $(\epsilon_g, \epsilon_H)$-approximate local minimum $\xb$ of $F(\xb)$ as follows
\begin{align}
    \|\dF(\xb)\|_2\leq \epsilon_g, \quad\mineig(\HF(\xb)) \geq -\epsilon_H,\label{eq:def_localminima}
\end{align}
where $\epsilon_g,\epsilon_H>0$ are predefined precision parameters. The most classic algorithm to find the approximate local minimum is cubic-regularized (CR) Newton method, which was originally proposed in the seminal paper by \citet{Nesterov2006Cubic}. Generally speaking, in the $k$-th iteration, cubic regularization method solves a subproblem, which minimizes a cubic-regularized second-order Taylor expansion at the current iterate $\xb_k$. The update rule can be written as follows:
\begin{align}
    &\hb_k = \argmin_{\hb \in \RR^d} \la \dF(\xb_k), \hb\ra   +1/2\la \HF(\xb_k)\hb, \hb\ra +M/6\|\hb\|_2^3,\label{def:total_sub_problem}\\
    &\xb_{k+1} = \xb_k+\hb_k, \label{def:total_sub_problem2}
\end{align}
where $M>0$ is a penalty parameter. \citet{Nesterov2006Cubic} proved that to find an $(\epsilon, \sqrt{\epsilon})$-approximate local minimum of a nonconvex function $F$, cubic regularization requires at most $O(\epsilon^{-3/2})$ iterations. However, when applying cubic regularization to nonconvex finite-sum optimization in \eqref{def:problem}, a major bottleneck of cubic regularization is that it needs to compute $n$ individual gradients $\nabla f_i(\xb_k)$ and Hessian matrices $\Hf_i(\xb_k)$ at each iteration, which leads to a total $O(n\epsilon^{-3/2})$ gradient complexity (i.e., number of queries to the stochastic gradient oracle $\df_i(\xb)$ for some $i$ and $\xb$) and $O(n\epsilon^{-3/2})$ Hessian  complexity (i.e., number of queries to the stochastic Hessian oracle $\nabla^2 f_i(\xb)$ for some $i$ and $\xb$). Such computational overhead will be extremely expensive when $n$ is large as is in many large-scale machine learning applications.

To 
overcome the aforementioned computational burden of cubic regularization, \citet{kohler2017sub, xu2017newton} used subsampled gradient and subsampled Hessian, which achieve $\tilde O(n\epsilon^{-3/2} \land \epsilon^{-7/2})$ gradient complexity and $\tilde O(n\epsilon^{-3/2} \land \epsilon^{-5/2})$ Hessian complexity. \citet{zhou2018stochastic} proposed a stochastic variance reduced cubic regularization method (SVRC), which uses novel semi-stochastic gradient and semi-stochastic Hessian estimators inspired by variance reduction for first-order finite-sum optimization \citep{johnson2013accelerating, Reddi2016Stochastic, allen2016variance}, which attains $O(n^{4/5}\epsilon^{-3/2})$ Second-order Oracle (SO) complexity\footnote{Second-order Oracle (SO) returns triple $[f_i(\xb), \nabla f_i(\xb), \nabla^2 f_i(\xb)]$ for some $i$ and $\xb$, hence the SO complexity can be seen as the maximum of gradient and Hessian complexities.}.  \citet{zhou2018sample, wang2018sample,zhang2018adaptive} used a simpler semi-stochastic gradient compared with \citep{zhou2018stochastic}, and semi-stochastic Hessian, which 
achieves a better Hessian complexity, i.e., $O(n^{2/3}\epsilon^{-3/2})$. 
However, it is unclear whether the gradient and Hessian complexities of the aforementioned SVRC algorithms can be further improved. Furthermore, all these algorithms need to use the semi-stochastic Hessian estimator, which is not compatible with Hessian-vector product-based cubic subproblem solvers \citep{Agarwal2017Finding, Carmon2016Gradient, NIPS2018_8269}. Therefore, the cubic subproblem \eqref{def:total_sub_problem2} in each iteration of existing SVRC algorithms has to be solved by computing the inverse of the Hessian matrix, whose computational complexity is at least $O(d^w)$\footnote{$w$ is the matrix multiplication constant, where $w = 2.37...$ \citep{Golub:1996:MC:248979}.
}. This makes existing SVRC algorithms not very practical for high-dimensional problems.

In this paper, we first show that the gradient and Hessian complexities of SVRC-type algorithms can be further improved. The core idea is to use novel recursively updated semi-stochastic gradient and Hessian estimators, which are inspired by the stochastic path-integrated differential
estimator (SPIDER) \citep{fang2018spider} and the StochAstic Recursive grAdient algoritHm (SARAH) \citep{nguyen2017sarah} for first-order optimization. We show that such kind of estimators can be extended to second-order optimization to reduce the Hessian complexity. Nevertheless, our analysis is very different from that in \cite{fang2018spider,nguyen2017sarah}, because we study a fundamentally different optimization problem (i.e., finding local minima against finding first-order stationary points) and a completely different optimization algorithm (i.e., cubic regularization versus gradient method). In addition, in order to 
reduce the runtime complexity of existing SVRC algorithms, we further propose a \emph{Hessian-free} SVRC method that can not only use the novel semi-stochastic gradient estimator, but also leverage the Hessian-vector product-based fast cubic subproblem solvers. 
Experiments on benchmark nonconvex finite-sum optimization problems illustrate the superiority of our newly proposed SVRC algorithms over the state-of-the-art (Due to space limit, we include the experiments in Appendix \ref{sec:experiment}). 

In detail, our contributions are summarized as follows:
\begin{enumerate}[leftmargin = *]
 \item We propose a new SVRC algorithm, namely $\nameheavy$, which can find an $(\epsilon, \sqrt{\epsilon})$-approximate local minimum with $\tilde O(n\epsilon^{-3/2} \land  \epsilon^{-3})$ gradient complexity and $\tilde O(n \land \epsilon^{-1} + n^{1/2}\epsilon^{-3/2} \land \epsilon^{-2})$ Hessian complexity. Compared with previous work in cubic regularization, the gradient and Hessian complexity of $\nameheavy$ is strictly better than the algorithms in \cite{zhou2018sample, wang2018sample, zhang2018adaptive}, and better than that in \cite{zhou2018stochastic,shen2019stochastic} in a wide regime.
    
\item We further propose a new algorithm $\namefree$, which requires $\tilde O(\epsilon^{-3} \land n\epsilon^{-2})$ stochastic gradient and Hessian-vector product computations to find an $(\epsilon, \sqrt{\epsilon})$-approximate local minimum. 
$\namefree$ is strictly better than the algorithms in \citep{Agarwal2017Finding, Carmon2016Gradient, tripuraneni2018stochastic} when $n \gg 1$. The runtime complexity of $\namefree$ is also better than that of $\nameheavy$ when the problem dimension $d$ is large.
\end{enumerate}

In an independent and concurrent work \citep{shen2019stochastic}, two stochastic trust region methods namely STR1 and STR2 were proposed, which are based on the same idea of variance reduction using SPIDER, and are related to our first algorithm $\nameheavy$. 
Our $\nameheavy$ is better than STR1 because it enjoys the same Hessian complexity but a better gradient complexity than STR1. Compared with STR2, our $\nameheavy$ has a consistently lower Hessian complexity and lower gradient complexity in a wide regime (i.e., $ \epsilon \gg n^{-1/2}$). Since Hessian complexity is the dominating term in cubic regularization method \citep{zhou2018sample, wang2018sample}, our $\nameheavy$ is arguably better than STR2, as verified by our experiments. 


For the ease of comparison, 
we summarize the comparison of methods which need to compute the Hessian explicitly in Table \ref{table:11}, the Hessian-free or Hessian-vector product based methods in Table \ref{table:22}.


\begin{table}[ht]
\caption{Comparisons of different methods to find an $(\epsilon, \sqrt{\rho\epsilon})$-local minimum on gradient and Hessian complexity.}\label{table:11}
\begin{center}
\begin{tabular}{ccc}
\toprule
Algorithm & Gradient& Hessian \\
\midrule
CR  & \multirow{2}{*}{$O\big(\frac{n}{\epsilon^{3/2}}\big)$} & \multirow{2}{*}{$O\big(\frac{n}{\epsilon^{3/2}}\big)$} \\
\small{\citep{Nesterov2006Cubic}} & & \\
SCR & \multirow{2}{*}{$\tilde O\big(\frac{n}{\epsilon^{3/2}}\land \frac{1}{\epsilon^{7/2}}\big)$}  & \multirow{2}{*}{$\tilde O\big(\frac{n}{\epsilon^{3/2}}\land \frac{1}{\epsilon^{5/2}}\big)$}  \\
\small{\citep{kohler2017sub,xu2017newton}} & & \\
SVRC & \multirow{2}{*}{$\tilde O\big( \frac{n^{4/5}}{\epsilon^{3/2}}\big)$} & \multirow{2}{*}{$\tilde O\big( \frac{n^{4/5}}{\epsilon^{3/2}}\big)$} \\
\small{\citep{zhou2018stochastic}} & & \\
(Lite-)SVRC & \multirow{2}{*}{$\tilde O\big(\frac{n}{\epsilon^{3/2}}\big)$} & \multirow{2}{*}{$\tilde O\big( \frac{n^{2/3}}{\epsilon^{3/2}}\big)$} \\
\small{\citep{zhou2018sample,wang2018sample, zhou2019stochastic}} & & \\
SVRC & \multirow{2}{*}{$O\big(\frac{n}{\epsilon^{3/2}}\land \frac{n^{2/3}}{\epsilon^{5/2}}\big)$} & \multirow{2}{*}{$O\big( \frac{n^{2/3}}{\epsilon^{3/2}}\big)$} \\
\small{\citep{zhang2018adaptive}} & & \\
STR1 & \multirow{2}{*}{$\tilde O\big(\frac{n}{\epsilon^{3/2}}\land \frac{n^{1/2}}{\epsilon^{2}}\big)$} & \multirow{2}{*}{$\tilde O\big( \frac{n^{1/2}}{\epsilon^{3/2}}\land \frac{1}{\epsilon^2}\big)$} \\
\small{\citep{shen2019stochastic}} & & \\
STR2 & \multirow{2}{*}{$\tilde O\big(\frac{n^{3/4}}{\epsilon^{3/2}}\big)$} & \multirow{2}{*}{$\tilde O\big( \frac{n^{3/4}}{\epsilon^{3/2}}\big)$} \\
\small{\citep{shen2019stochastic}} & & \\
$\nameheavy$ & \multirow{2}{*}{$\tilde O\big(\frac{n}{\epsilon^{3/2}}\land \frac{n^{1/2}}{\epsilon^2} \land \frac{1}{\epsilon^3}\big)$} & \multirow{2}{*}{$\tilde O\big( \frac{n^{1/2}}{\epsilon^{3/2}}\land \frac{1}{\epsilon^2}\big)$} \\
\small{(This work)} & & \\
\bottomrule
\end{tabular}
\end{center}
\end{table}

\begin{table}[ht]
\caption{Comparisons of different methods to find an $(\epsilon, \sqrt{\rho\epsilon})$-local minimum both on stochastic gradient and Hessian-vector product computations.}\label{table:22}
\begin{center}
\begin{tabular}{cc}
\toprule
Algorithm & Gradient $\&$ Hessian-vector product \\
\midrule
SGD  & \multirow{2}{*}{$\tilde O\big(\frac{1}{\epsilon^{7/2}}\big)$}  \\
\small{\citep{fang2019sharp}} &  \\
SGD  & \multirow{2}{*}{$\tilde O\big(\frac{1}{\epsilon^{4}}\big)$}  \\
\small{\citep{jin2019stochastic}} &  \\
Fast-Cubic  & \multirow{2}{*}{$ \tilde O\big(\frac{n}{\epsilon^{3/2}} + \frac{n^{3/4}}{\epsilon^{7/4}}\big)$}  \\
\small{\citep{Agarwal2017Finding}} &  \\
GradientCubic  & \multirow{2}{*}{$ \tilde O\big(\frac{n}{\epsilon^2}\big)$}  \\
\small{\citep{Carmon2016Gradient}} &  \\
STC  & \multirow{2}{*}{$\tilde O\big(\frac{1}{\epsilon^{7/2}}\big)$}  \\
\small{\citep{tripuraneni2018stochastic}} &  \\
SPIDER  & \multirow{2}{*}{$\tilde O\big((\frac{\sqrt{n}}{\epsilon^2} + \frac{1}{\epsilon^{2.5}}) \land \frac{1}{\epsilon^3}\big)$}  \\
\small{\citep{fang2018spider}} &  \\
$\namefree$  & \multirow{2}{*}{$\tilde O\big(\frac{n}{\epsilon^{2}}\land \frac{1}{\epsilon^3}\big)$}  \\
\small{(This work)} &  \\
\bottomrule
\end{tabular}
\end{center}
\end{table}

\footnotetext[3]{The complexity for Natasha2 to find an $(\epsilon, \epsilon^{1/4})$-local minimum only requires $\tilde O(\epsilon^{-3.25})$. Here we adapt the complexity result for finding an $(\epsilon, \epsilon^{1/2})$-approximate local minimum.}

\section{Additional Related Work}
In this section, we review additional related work that is not discussed in the introduction section.

\noindent\textbf{Cubic Regularization and Trust-Region Methods}
Since cubic regularization was first proposed by \citet{Nesterov2006Cubic}, there has been a line of followup research. It was extended to adaptive regularized cubic methods (ARC) by \citet{Cartis2011Adaptive, Cartis2011Adaptive2}, which enjoy the same iteration complexity as standard cubic regularization while having better empirical performance. The first attempt to make cubic regularization a Hessian-free method was done by \citet{Carmon2016Gradient}, which solves the cubic sub-problem by gradient descent, requiring in total $\tilde O(n\epsilon^{-2})$ stochastic gradient and Hessian-vector product computations. \citet{Agarwal2017Finding} solved cubic sub-problem by fast matrix inversion based on accelerated gradient descent, which requires $\tilde O(n\epsilon^{-3/2}+ n^{3/4}\epsilon^{-7/4})$ stochastic gradient and Hessian-vector product computations. In the pure stochastic optimization setting, \citet{tripuraneni2018stochastic} proposed stochastic cubic regularization method, which uses subsampled gradient and Hessian-vector product-based cubic subproblem solver, and requires $\tilde O(\epsilon^{-3.5})$ stochastic gradient and Hessian-vector product computations. A closely related second-order method to cubic regularization methods are trust-region methods \citep{conn2000trust,cartis2009trust,cartis2012complexity, cartis2013evaluation}. Recent studies \citep{blanchet2016convergence,curtis2017trust, martinez2017cubic} proved that the trust-region method can achieve the same iteration complexity as the cubic regularization method. \citet{xu2017newton} also extended trust-region method to subsampled trust-region method for nonconvex finite-sum optimization. 

\noindent\textbf{Local Minima Finding}
Besides cubic regularization and trust-region type methods, there is another line of research for finding approximate local minima, which is based on first-order optimization. \citet{ge2015escaping, jin2017escape} proved that (stochastic) gradient methods with additive noise are able to escape from nondegenerate saddle points and find approximate local minima.  \citet{Carmon2016Accelerated, royer2017complexity, allen2017natasha,xu2018first,allen2018neon2,jin2017accelerated,yu2017saving,yu2017third,zhou2018finding,fang2018spider} showed that by alternating first-order optimization and Hessian-vector product based negative curvature descent, one can find approximate local minima even more efficiently. Very recently, \citet{fang2019sharp,jin2019stochastic} showed that stochastic gradient descent itself can escape from saddle points.

\noindent\textbf{Variance Reduction}
Variance reduction techniques play an important role in our proposed algorithms. Variance reduction techniques were first proposed for convex finite-sum optimization, which use semi-stochastic gradient to reduce the variance of the stochastic gradient and improve the gradient complexity. 
Representative algorithms include Stochastic Average Gradient (SAG) \citep{roux2012stochastic}, Stochastic Variance Reduced Gradient (SVRG) \citep{johnson2013accelerating,xiao2014proximal}, SAGA \citep{defazio2014saga} and SARAH \citep{nguyen2017sarah}, to mention a few. For nonconvex finite-sum optimization problems, 
\citet{garber2015fast,shalev2016sdca} studied the case where each individual function is nonconvex, but their sum is still (strongly) convex. \citet{Reddi2016Stochastic,allen2016variance} extended SVRG to noncovnex finite-sum optimization, which is able to converge to first-order stationary point with better gradient complexity than vanilla gradient descent.  \citet{fang2018spider,zhou2018stochastic_nips, wang2018spiderboost, nguyen2019optimal} further improved the gradient complexity for nonconvex finite-sum optimization to be (near) optimal. 

\section{Notation and Preliminaries}
In this work, all index subsets are multiset. We use $\df_\cI(\xb)$ to represent $1/|\cI|\cdot \sum_{i \in \cI} \df_i (\xb)$ if $|\cI| < n$ and $\dF(\xb)$ otherwise. We use $\Hf_\cI(\xb)$ to represent $1/|\cI|\cdot \sum_{i \in \cI} \Hf_i (\xb)$ if $|\cI| < n$ and $\HF(\xb)$ otherwise. For a vector $\vb$, we denote its $i$-th coordinate by $v_i$. We denote vector Euclidean norm by $\| \vb\|_2$. For any matrix $\Ab$, we denote its $(i,j)$ entry by $A_{i,j}$, its Frobenius norm by $\|\Ab\|_F$ , and its spectral norm by $\|\Hb\|_2$. For a symmetric matrix $\Hb\in\RR^{d\times d}$, we denote its minimum eigenvalue by $\mineig(\Hb)$. For symmetric matrices $\Ab, \Bb \in \RR^{d \times d}$, we say $\Ab \succeq \Bb$ if $\mineig(\Ab - \Bb) \geq 0$. We use $f_n=O(g_n)$ to denote that $f_n\le C g_n$ for some constant $C>0$ and use $f_n=\tilde O(g_n)$ to hide the logarithmic factors of $g_n$. We use $a \land b = \min\{a,b\}$. 

We begin with a few assumptions that are needed for later theoretical analyses of our algorithms.

The following assumption says that the gap between the function value at the initial point $\xb_0$ and the minimal function value is bounded.
\begin{assumption}\label{assumption_1}
For any function $F(\xb)$ and an initial point $\xb_0$, there exists a constant $0<\Delta_F<\infty$ such that $F(\xb_0) - \inf_{\xb \in \RR^d} F(\xb)\leq\Delta_F$.
\end{assumption}

We also need the following $L$-gradient Lipschitz and $\rho$-Hessian Lipschitz assumption.  

\begin{assumption}\label{assumption_2}
For each $i$, we assume that $f_i$ is $L$-gradient Lipschitz continuous and $\rho$-Hessian Lipschitz continuous, where we have $\|\df_i(\xb) - \df_i(\yb)\|_2 \leq L\|\xb - \yb\|_2$ and $\|\Hf_i(\xb) - \Hf_i(\yb)\|_2 \leq \rho \|\xb - \yb\|_2$ for all $\xb, \yb \in \RR^d$.
\end{assumption}

Note that $L$-gradient Lipschitz is not required in the original cubic regularization algorithm \citep{Nesterov2006Cubic} and the SVRC algorithm \citep{zhou2018stochastic}. However, for most other SVRC algorithms \citep{zhou2018sample, wang2018sample, zhang2018adaptive}, they need the $L$-gradient Lipschitz assumption. 

In addition, we need the difference between the stochastic gradient and the full gradient to be bounded.

\begin{assumption}\label{assumption_3}
We assume that $F$ has $\upp$-bounded stochastic gradient, where we have $\|\nabla f_i(\xb) - \nabla F(\xb)\|_2 \leq \upp, \forall \xb \in \RR^d, \forall i \in [n]$.
\end{assumption}
It is worth noting that Assumption \ref{assumption_3} is weaker than the assumption that each $f_i$ is Lipschitz continuous, which has been made in \cite{kohler2017sub, zhou2018sample,wang2018sample,zhang2018adaptive}. We would also like to point out that we can make additional assumptions on the variances of the stochastic gradient and Hessian, such as the ones made in \cite{tripuraneni2018stochastic}. Nevertheless, making these additional assumptions 
does not improve the dependency of the gradient and Hessian complexities or the stochastic gradient and Hessian-vector product computations on $\epsilon$ and $n$. Therefore we chose not making these additional assumptions on the variances.

\begin{algorithm*}[t]
\caption{Stochastic Recursive Variance-Reduced Cubic Regularization ($\nameheavy$)}\label{algorithm:1}
\begin{algorithmic}[1]
  \STATE \textbf{Input:} Total iterations $T$, batch sizes $\{B_t^{(g)}\}_{t=1}^T, \{B_t^{(h)}\}_{t=1}^T$, cubic penalty parameter $\{M_{t}\}_{t=1}^T$, inner gradient length $S^{(g)}$, inner Hessian length $S^{(h)}$, initial point $\xb_0$, accuracy $\epsilon$ and Hessian Lipschitz constant $\rho$.
  \FOR{$t=0,\ldots,T-1$}
  \STATE Sample index set $\cJ_t$ with $|\cJ_t| = B_t^{(g)};$ $\cI_t$ with $|\cI_t| = B_t^{(h)};$
  \begin{align}
      \vb_t &\leftarrow  \begin{cases}
               \nabla f_{\cJ_t}(\xb_t), &\mod (t,S^{(g)}) = 0\\
               \nabla f_{\cJ_t}(\xb_t) - \nabla f_{\cJ_t}(\xb_{t-1}) +\vb_{t-1}, &\text{else}
            \end{cases}\label{def_vv}\\
        \Ub_t &\leftarrow\begin{cases}
               \nabla^2 f_{\cI_t}(\xb_t), &\mod (t,S^{(h)}) = 0\\
               \nabla^2 f_{\cI_t}(\xb_t) - \nabla^2 f_{\cI_t}(\xb_{t-1}) +\Ub_{t-1}, &\text{else}
            \end{cases}\label{def_uu}\\
            \hb_t &\leftarrow \argmin_{\hb\in \RR^d} m_t(\hb):= \la \vb_t,\hb\ra + \frac{1}{2}\la\Ub_t \hb,\hb\ra + \frac{M_{t}}{6} \|\hb\|_2^3\label{stochas_subproblem}
  \end{align}
  \STATE $\xb_{t+1} \leftarrow \xb_{t} + \hb_t$
  \IF {$\|\hb_t\|_2 \leq \sqrt{\epsilon/\rho}$}
  \STATE  \textbf{return} $\xb_{t+1}$
  \ENDIF

  \ENDFOR
\end{algorithmic}

\end{algorithm*}

\section{The Proposed $\nameheavy$ Algorithm}\label{sec:SVRC}



In this section, we present $\nameheavy$, a novel algorithm which utilizes new semi-stochastic gradient and Hessian estimators compared with previous SVRC algorithms. We also provide a convergence analysis of the proposed algorithm. 

\subsection{Algorithm Description}

In order to reduce the computational complexity for calculating full gradient and full Hessian in \eqref{def:total_sub_problem}, several ideas such as subsampled/stochastic gradient and Hessian \citep{kohler2017sub, xu2017newton, tripuraneni2018stochastic} and variance-reduced semi-stochastic gradient and Hessian \citep{zhou2018stochastic, wang2018sample, zhang2018adaptive} have been used in previous work. $\nameheavy$ follows this line of work. The key idea is to use a new construction of semi-stochastic gradient and Hessian estimators, which are recursively updated in each iteration, and reset periodically after certain number of iterations (i.e., an epoch). This is inspired by the first-order variance reduction algorithms SPIDER \citep{fang2018spider} and SARAH \citep{nguyen2017sarah}. 
$\nameheavy$ constructs semi-stochastic gradient and Hessian as in \eqref{def_vv} and \eqref{def_uu} respectively.
To be more specific, in the $t$-th iteration when $\mod(t,S^{(g)}) = 0$ or $\mod(t,S^{(h)}) = 0$, where $S^{(g)}, S^{(h)}$ are the epoch lengths of gradient and Hessian, $\nameheavy$ will set the semi-stochastic gradient $\vb_t$ and Hessian $\Ub_t$ to be a subsampled gradient $\df_{\cJ_t}(\xb_t)$ and Hessian $\Hf_{\cJ_t}(\xb_t)$ at point $\xb_t$, respectively. 
In the $t$-th iteration when $\mod(t,S) \neq 0$ or $\mod(t,S^{(h)}) \neq 0$, $\nameheavy$ constructs semi-stochastic gradient and Hessian $\vb_t$ and $\Ub_t$ based on previous estimators $\vb_{t-1}$ and $\Ub_{t-1}$ recursively. 
With semi-stochastic gradient $\vb_t$, semi-stochastic Hessian $\Ub_t$ and $t$-th Cubic penalty parameter $M_t$, $\nameheavy$ constructs the $t$-th Cubic subproblem $m_t$ and solves for the solution to $m_t$ as $t$-th update direction as \eqref{stochas_subproblem}.
If $\|\hb_t\|_2$ is less than a given threshold which we set it as $\sqrt{\epsilon/\rho}$, $\nameheavy$ returns $\xb_{t+1} = \xb_t + \hb_t$ as its output. Otherwise, $\nameheavy$ updates $\xb_{t+1} = \xb_t + \hb_t$ and continues the loop.

The main difference between $\nameheavy$ and previous stochastic cubic regularization algorithms \citep{kohler2017sub, xu2017newton, zhou2018stochastic, zhou2018sample, wang2018sample, zhang2018adaptive} is that $\nameheavy$ adapts new semi-stochastic gradient and semi-stochastic Hessian estimators, which are defined recursively and have smaller asymptotic variance. The use of such semi-stochastic gradient has been proved to help reduce the gradient complexity in first-order nonconvex finite-sum optimization for finding stationary points \citep{fang2018spider,wang2018spiderboost,nguyen2019optimal}. Our work takes one step further to apply it to Hessian, and we will later show that it helps reduce the gradient and Hessian complexities in second-order nonconvex finite-sum optimization for finding local minima. 

\subsection{Convergence Analysis}\label{sec:rate_for_heavy}
In this subsection, we present our theoretical results about $\nameheavy$. 
While the idea of using variance reduction technique for cubic regularization is hardly new,  the new semi-stochastic gradient and Hessian estimators in \eqref{def_vv} and \eqref{def_uu} bring new technical challenges in the convergence analysis.

To describe whether a point $\xb$ is a local minimum, we follow the original cubic regularization work \citep{Nesterov2006Cubic} to use the following criterion $\mu(\xb)$:
\begin{definition}\label{main0_mu}
For any $\xb$, define $\mu(\xb)$ as $\mu(\xb) = \max\{\|\nabla F(\xb)\|_2^{3/2},-\lambda_{\min}^3\big(\nabla^2 F(\xb)\big)/{\rho^{3/2}}\}$.
\end{definition}
It is easy to note that $\mu(\xb) \leq \epsilon^{3/2}$ if and only if $\xb$ is an $(\epsilon, \sqrt{\rho\epsilon})$-approximate local minimum. Thus, in order to find an $(\epsilon, \sqrt{\rho\epsilon})$-approximate local minimum, it suffices to find a point $\xb$ which satisfies $\mu(\xb) \leq \epsilon^{3/2}$.

The following theorem provides the convergence guarantee of $\nameheavy$ for finding an $(\epsilon, \sqrt{\rho\epsilon})$-approximate local minimum.

\begin{theorem}\label{Theorem1}
Under Assumptions \ref{assumption_1}, \ref{assumption_2} and \ref{assumption_3},  set the cubic penalty parameter $M_t =4\rho$ for any $t$ and the total iteration number $T \geq 40 \Delta_F\rho^{1/2}\epsilon^{-3/2}$. 
For $t$ such that $\mod(t,S^{(g)}) \neq 0$ or $\mod(t,S^{(h)}) \neq 0$, set the gradient sample size $B_t^{(g)}$ and Hessian sample size $B_t^{(h)}$ as
\begin{align}
    B^{(g)}_t &\geq n \land  \frac{1440L^2S^{(g)}\|\hb_{t-1}\|_2^2\log^2(2T/\xi)}{\epsilon^2},\label{gradientVariance-1}\\
     B^{(h)}_t &\geq n \land  \frac{800\rho S^{(h)}\|\hb_{t-1}\|_2^2\log^2(2Td/\xi)}{\epsilon} \label{hessianVariance-1}.
\end{align}
For $t$ such that $\mod(t,S^{(g)}) = 0$ or $\mod(t,S^{(h)}) = 0$, set the gradient sample size $B_t^{(g)}$ and Hessian sample size $B_t^{(h)}$ as
\begin{align}
    B^{(g)}_t &\geq n \land  \frac{1440\upp^2\log^2(2 T   /\xi)}{\epsilon^2},\label{gradientVariance-2}\\
    B^{(h)}_t &\geq n \land \frac{800L^2\log^2(2Td/\xi)}{\rho\epsilon}\label{hessianVariance-2}.
\end{align}
Then with probability at least $1-\xi$, $\nameheavy$ outputs $\xb_{\text{out}}$ satisfying $\mu(\xb_{\text{out}}) \leq 600 \epsilon^{3/2}$, i.e., an $(\epsilon, \sqrt{\rho \epsilon})$-approximate local minimum.  
\end{theorem}

Next corollary spells out the exact gradient complexity and Hessian complexity of $\nameheavy$ to find an $(\epsilon, \sqrt{\rho \epsilon})$-approximate local minimum. 
\begin{corollary}\label{corollary_1}
Under the same conditions as Theorem \ref{Theorem1}, if we set $S^{(g)}, S^{(h)}$ as $S^{(g)} =  \sqrt{\rho\epsilon}/L\cdot\sqrt{n \land M^2/\epsilon^2}$ and $S^{(h)} = \sqrt{n \land L/(\rho\epsilon)}$, 
and set $T, \{B^{(g)}_t\}, \{B^{(h)}_t\}$ as their lower bounds in \eqref{gradientVariance-1}- \eqref{hessianVariance-2}, then with probability at least $1-\xi$, $\nameheavy$ will output an $(\epsilon, \sqrt{\rho \epsilon})$-approximate local minimum within 
\begin{align}
    \tilde O\bigg(n\land \frac{L^2}{\rho\epsilon} + \frac{\sqrt{\rho}\Delta_F}{\epsilon^{3/2}}\sqrt{n\land \frac{L^2}{\rho\epsilon}}\bigg)\notag
\end{align}
stochastic Hessian evaluations and
\begin{align}
    \tilde O\bigg(n\land \frac{\upp^2}{\epsilon^2} + \frac{\Delta_F}{\epsilon^{3/2}}\bigg[\sqrt{\rho}n \land \frac{L \sqrt{n} }{\sqrt{\epsilon}} \land \frac{L M }{\epsilon^{3/2}}\bigg]\bigg)\notag
\end{align}
stochastic gradient evaluations. 
\end{corollary}
\begin{remark}
For $\nameheavy$, if we assume $\upp, L, \rho, \Delta_F$ to be constants, then its gradient complexity is $\tilde O(n/\epsilon^{3/2} \land  \sqrt{n}/\epsilon^2\land \epsilon^{-3} )$, 
and its Hessian complexity is $\tilde O(n\land \epsilon^{-1} + n^{1/2}\epsilon^{-3/2} \land \epsilon^{-2})$.
Regarding Hessian complexity, suppose that $\epsilon \ll 1$, then the Hessian complexity of $\nameheavy$ can be simplified as $\tilde O (n^{1/2}\epsilon^{-3/2} \land \epsilon^{-2})$. Compared with existing SVRC algorithms \citep{zhou2018sample,zhang2018adaptive,wang2018sample}, $\nameheavy$ outperforms the best-known Hessian sample complexity by a factor of $n^{1/6} \land n^{2/3}\epsilon^{1/2}$. In terms of gradient complexity, $\nameheavy$ outperforms STR2 \citep{shen2019stochastic} by a factor of $n^{3/4}\epsilon^{3/2}$ when $\epsilon \gg n^{-1/2}$.
\end{remark}
\begin{remark}
Note that both Theorem \ref{Theorem1} and Corollary \ref{corollary_1} still hold when Assumption \ref{assumption_3} does not hold. In that case, $M = \infty$ and $\nameheavy$'s Hessian complexity remains the same, while its gradient complexity can be potentially worse, i.e., $\tilde O(n/\epsilon^{3/2} \land  \sqrt{n}/\epsilon^2)$, which degenerates to that of STR1 \citep{shen2019stochastic}. 
\end{remark}

\section{Hessian-Free $\nameheavy$}\label{sec:hessfree}
\begin{algorithm*}[h]
\caption{Hessian Free Stochastic Recursive Variance-Reduced Cubic Regularization ($\namefree$)}\label{algorithm:2}
\begin{algorithmic}[1]
  \STATE \textbf{Input:} Total iterations $T$, batch sizes $\{B_t^{(g)}\}_{t=1}^T, \{B_t^{(h)}\}_{t=1}^T$, cubic penalty parameter $\{M_{t}\}_{t=1}^T$, inner gradient length $S^{(g)}$, initial point $\xb_0$, accuracy $\epsilon$, Hessian Lipschitz constant $\rho$, gradient Lipschitz constant $L$ and failure probability $\xi$.
  \FOR{$t=0,\ldots,T-1$}
   \STATE Sample index set $\cJ_t, |\cJ_t| = B_t^{(g)};$ $\cI_t, |\cI_t| = B_t^{(h)};$
 \begin{align}
      \vb_t &\leftarrow  \begin{cases}
               \nabla f_{\cJ_t}(\xb_t), &\mod (t,S^{(g)}) = 0\\
               \nabla f_{\cJ_t}(\xb_t) - \nabla f_{\cJ_t}(\xb_{t-1}) +\vb_{t-1},& \text{else}
            \end{cases},
            \Ub_t[\cdot] \leftarrow  \nabla^2 f_{\cI_t}(\xb_t)[\cdot]\notag
 \end{align}
  \STATE $\hb_t \leftarrow \namesub(\Ub_t[\cdot],\vb_t,  M_t,1/(16L), \sqrt{\epsilon/\rho}, 0.5,\xi/(3T))$ \COMMENT{See Algorithm \ref{alg:sub} in Appendix \ref{appendix: add_apgorithm}}
  \IF{$m_t(\hb_t) < -4\rho^{-1/2}\epsilon^{3/2}$}
  \STATE $\xb_{t+1} \leftarrow \xb_t + \hb_t$
  \ELSE
  \STATE $\hb_t \leftarrow \namefinal(\Ub_t[\cdot],\vb_t,  M_t, 1/(16L), \epsilon)$
  \COMMENT{See Algorithm \ref{alg:finalsub} in Appendix \ref{appendix: add_apgorithm}}
  \STATE \textbf{return} $\xb_{t+1} \leftarrow \xb_t + \hb_t$
  \ENDIF

  \ENDFOR
  
\end{algorithmic}
\end{algorithm*}

While $\nameheavy$ adapts novel semi-stochastic gradient and Hessian estimators to reduce both the gradient and Hessian complexities, 
it has three limitations for high-dimensional problems with $d \gg 1$: 
(1) it needs to compute and store the Hessian matrix, which needs $O(d^2)$ computational time and storage space;  (2) it needs to solve cubic subproblem $m_t$ exactly, which requires $O(d^{w})$ computational time because it needs to compute the inverse of a Hessian matrix \citep{Nesterov2006Cubic}; and (3) it cannot leverage the Hessian-vector product-based cubic subproblem solvers \citep{Agarwal2017Finding, Carmon2016Gradient, NIPS2018_8269} because of the use of the semi-stochastic Hessian estimator. It is interesting to ask whether we can modify $\nameheavy$ to overcome these shortcomings.

\subsection{Algorithm Description}


We present a Hessian-free algorithm $\namefree$ to address above limitations of $\nameheavy$ for high-dimensional problems, which only requires stochastic gradient and Hessian-vector product computations. $\namefree$ uses the same semi-stochastic gradient $\vb_t$ as $\nameheavy$. As opposed to $\nameheavy$ which has to construct semi-stochastic Hessian explicitly, $\namefree$ only accesses to stochastic Hessian-vector product. 
In detail, at each iteration $t$, $\namefree$ subsamples an index set $\cI_t$ and defines a stochastic Hessian-vector product function $\Ub_t[\cdot]:\RR^d \rightarrow \RR^d $ as follows:
\begin{align}
    \Ub_t[\vb] &= \Hf_{\cI_t}(\xb_t)[\vb], &\forall \vb \in \RR^d.\notag 
\end{align}
Note that although the subproblem depends on $\Ub_t$, $\namefree$ never explicitly computes this matrix. Instead, it only provides the subproblem solver access to $\Ub_t$ through stochastic Hessian-vector product function $\Ub_t[\cdot]$. The subproblem solver performs gradient-based optimization to solve the subproblem $m_t(\hb)$ as $\nabla m_t(\hb)$ depends on $\Ub_t$ only via $\Ub_t[\hb]$.
In detail, following \cite{tripuraneni2018stochastic}, $\namefree$ uses $\namesub$ (See Algorithms \ref{alg:sub} and \ref{alg:finalsub} in Appendix \ref{appendix: add_apgorithm}) and $\namefinal$ from \citep{Carmon2016Gradient}, to find an approximate solution $\hb_t$ to the cubic subproblem in \eqref{stochas_subproblem}. 
Both $\namesub$ and $\namefinal$ only need to access gradient $\vb_t$ and Hessian-vector product function $\Ub_t[\cdot]$ along with other problem-dependent parameters. With the output $\hb_t$ from $\namesub$, $\namefree$ decides either to update $\xb_t$ as $\xb_{t+1} \leftarrow \xb_t + \hb_t$ or to exit the loop. For the later case, $\namefree$ will call $\namefinal$ to output $\hb_t$, and takes $\xb_{t+1} = \xb_t + \hb_t$ as its final output. 


The main differences between $\nameheavy$ and $\namefree$ are two-fold. First, $\namefree$ only needs to compute stochastic gradient and Hessian-vector product. Since both of these two computations only take $O(d)$ time in many applications in machine learning, $\namefree$ is suitable for high-dimensional problems. In the sequel, following \citet{Agarwal2017Finding, Carmon2016Accelerated, tripuraneni2018stochastic}, we do not distinguish stochastic gradient and Hessian-vector product computations and consider them to have the same runtime complexity. Second, instead of solving cubic subproblem $m_t$ exactly, $\namefree$ adopts approximate subproblem solver $\namesub$ and $\namefinal$, both of which only need to access gradient and Hessian-vector product function, and again only take $O(d)$ time. Thus, $\namefree$ is computational more efficient than $\nameheavy$ when $d \gg 1$.

\subsection{Convergence Analysis}
We now provide the convergence guarantee of $\namefree$, which ensures that $\namefree$ will output an $(\epsilon, \sqrt{\rho\epsilon})$-approximate local minimum.
\begin{theorem}\label{thm:2}
Under Assumptions \ref{assumption_1}, \ref{assumption_2}, \ref{assumption_3}, suppose $\epsilon < L/(4\rho)$. Set the cubic penalty parameter $M_t =4\rho$ for any $t$ and the total iteration number $T \geq 25\Delta_F\rho^{1/2}\epsilon^{-3/2}$. Set the Hessian-vector product sample size $B_t^{(h)}$ as
\begin{align}
    B^{(h)}_t &\geq n\land \frac{1200L^2\log^2(3Td/\xi)}{\rho\epsilon}.\label{HessianVariance-3} 
\end{align}
For $t$ such that $\mod(t, S^{(g)}) \neq 0$, set the gradient sample size $B^{(g)}_t$ as
\begin{align}
    B^{(g)}_t \geq n\land \frac{2640L^2S^{(g)}\| \hb_{t-1}\|_2^2\log^2(3T/\xi)}{\epsilon^2}.\label{gradientVariance-3}
\end{align}
For $t$ such that $\mod(t,S^{(g)}) = 0$, set the gradient sample size $B^{(g)}_t$ as
\begin{align}
    B^{(g)}_t &\geq n\land \frac{2640\upp^2\log^2(3T/\xi)}{\epsilon^2}.\label{gradientVariance-4}
\end{align}
Then with probability at least $1-\xi$, $\namefree$ outputs $\xb_{\text{out}}$ satisfying $\mu(\xb_{\text{out}}) \leq 1300 \epsilon^{3/2}$, i.e., an $(\epsilon, \sqrt{\rho \epsilon})$-approximate local minimum. 
\end{theorem}
The following corollary calculates the total amount of stochastic gradient and Hessian-vector product computations of $\namefree$ to find an $(\epsilon, \sqrt{\rho\epsilon})$-approximate local minimum. 

\begin{corollary}\label{free_total_count}
Under the same conditions as Theorem \ref{thm:2}, if set $ S^{(g)} = \sqrt{\rho\epsilon}/L\cdot\sqrt{n \land M^2/\epsilon^2}$
and set $T, \{B^{(g)}_t\}, \{B^{(h)}_t\}$ as their lower bounds in \eqref{HessianVariance-3}-\eqref{gradientVariance-4}, then with probability at least $1-\xi$, $\namefree$ will output an $(\epsilon, \sqrt{\rho\epsilon})$-approximate local minimum within
\begin{align}
    & \tilde O\bigg[
    \bigg(n \land \frac{\upp^2}{\epsilon^2}\bigg)+\frac{\Delta_F }{\epsilon^{3/2}}\bigg(\sqrt{\rho}n\land \frac{L\sqrt{n}}{\sqrt{\epsilon}} \land \frac{LM}{\epsilon^{3/2}}\bigg)   + \bigg(\frac{L\Delta_F}{\epsilon^2} + \frac{L}{\sqrt{\rho\epsilon}}\bigg)\cdot \bigg(n\land \frac{L^2}{\rho\epsilon}\bigg) 
    \bigg]\label{free_total_count_1}
\end{align}
stochastic gradient and Hessian-vector product computations. 
\end{corollary}

\begin{remark}
For $\namefree$, if we assume $\rho, L, M, \Delta_F$ are constants, then \eqref{free_total_count_1} is $\tilde O(n\epsilon^{-2} \land \epsilon^{-3})$. 
For stochastic algorithms, the regime $n \rightarrow \infty$ is of most interest. In this regime, \eqref{free_total_count_1} becomes $\tilde O(\epsilon^{-3})$. Compared with other local minimum finding algorithms based on stochastic gradient and Hessian-vector product, $\namefree$ outperforms the results achieved by \citet{tripuraneni2018stochastic} and \citet{allen2018natasha2} by a factor of $\epsilon^{-1/2}$. $\namefree$ also matches the best-known result achieved by a recent first-order algorithm proposed in \citep{fang2018spider}. Note that the algorithm proposed by \citet{fang2018spider} needs to alternate the first-order finite-sum optimization algorithm SPIDER and negative curvature descent. In sharp contrast, $\namefree$ is a pure cubic regularization type algorithm and does not need to calculate the negative curvature direction.
\end{remark}
\begin{remark}
It is worth noting that both Theorem \ref{thm:2} and Corollary \ref{free_total_count} still hold when Assumption \ref{assumption_3} does not hold, and $\namefree$'s runtime complexity remains the same. The only difference is: without Assumption~\ref{assumption_3}, we need to use full gradient (i.e., $B_t^{(g)} = n$) instead of subsampled gradient at each iteration $t$. 
\end{remark}


\subsection{Discussions on runtime complexity}\label{sec:heavy_vs_free}
We would like to further compare the runtime complexity between $\nameheavy$ and $\namefree$. In specific, $\nameheavy$ needs $O(d)$ time to construct semi-stochastic gradient and $O(d^2)$ time to construct semi-stochastic Hessian. $\nameheavy$ also needs $O(d^w)$ time to solve cubic subproblem $m_t$ for each iteration. Thus, with the fact that the total number of  iterations is $T = O(\epsilon^{-3/2})$ by Corollary \ref{corollary_1}, $\nameheavy$ needs 
\begin{align}
    \tilde O\bigg(d\Big[\frac{n}{\epsilon^{3/2}}\land \frac{1}{\epsilon^3}\Big] + d^2 \Big[n\land\frac{1}{\epsilon} + \frac{\sqrt{n}}{\epsilon^{3/2}} \land \frac{1}{\epsilon^2}\Big] + \frac{d^w}{\epsilon^{3/2}}\bigg)\notag
\end{align}
runtime to find an $(\epsilon, \sqrt{\epsilon})$-approximate local minimum if we regard $M, L, \rho, \Delta_F$ as constants. As we mentioned before, for many machine learning problems, both stochastic gradient and Hessian-vector product computations only need $O(d)$ time, therefore the runtime of $\namefree$ is $\tilde O(dn\epsilon^{-2}\land d\epsilon^{-3})$. We conclude that $\namefree$ outperforms $\nameheavy$ when $d$ is large, which is in accordance with the fact that Hessian-free methods are superior for high dimension machine learning tasks. On the other hand, a careful calculation can show that the runtime of $\nameheavy$  can be less than that of $\namefree$ when $d$ is moderately small. This is also reflected in our experiments in Section \ref{sec:experiment}.


 
\section{Experiments}\label{sec:experiment}
\begin{figure*}[h]
	\begin{center}
		\subfigure[a9a]{\includegraphics[width=0.32\linewidth]{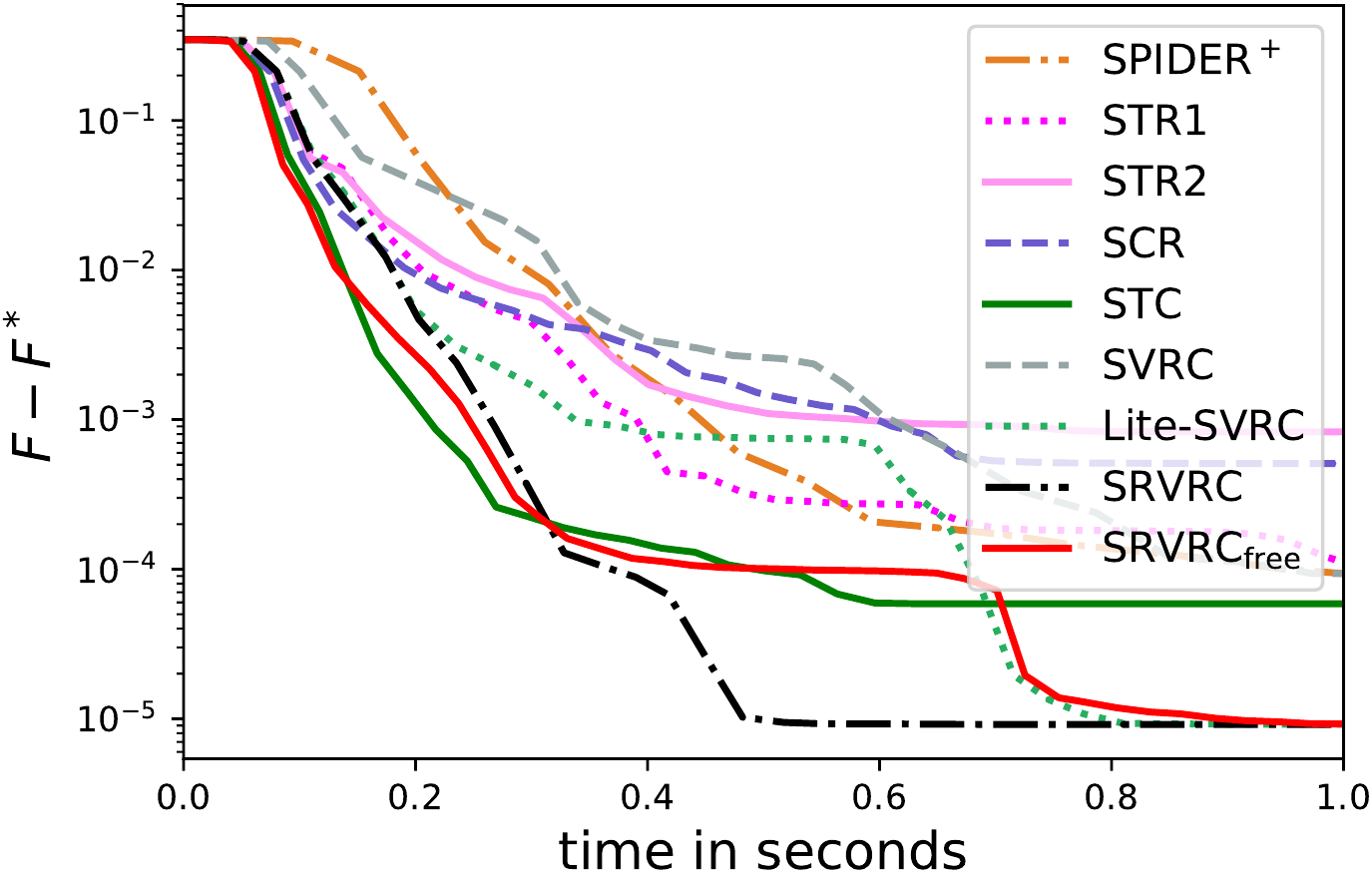}
		\label{fig:a9a}}
		\subfigure[covtype]{\includegraphics[width=0.32\linewidth]{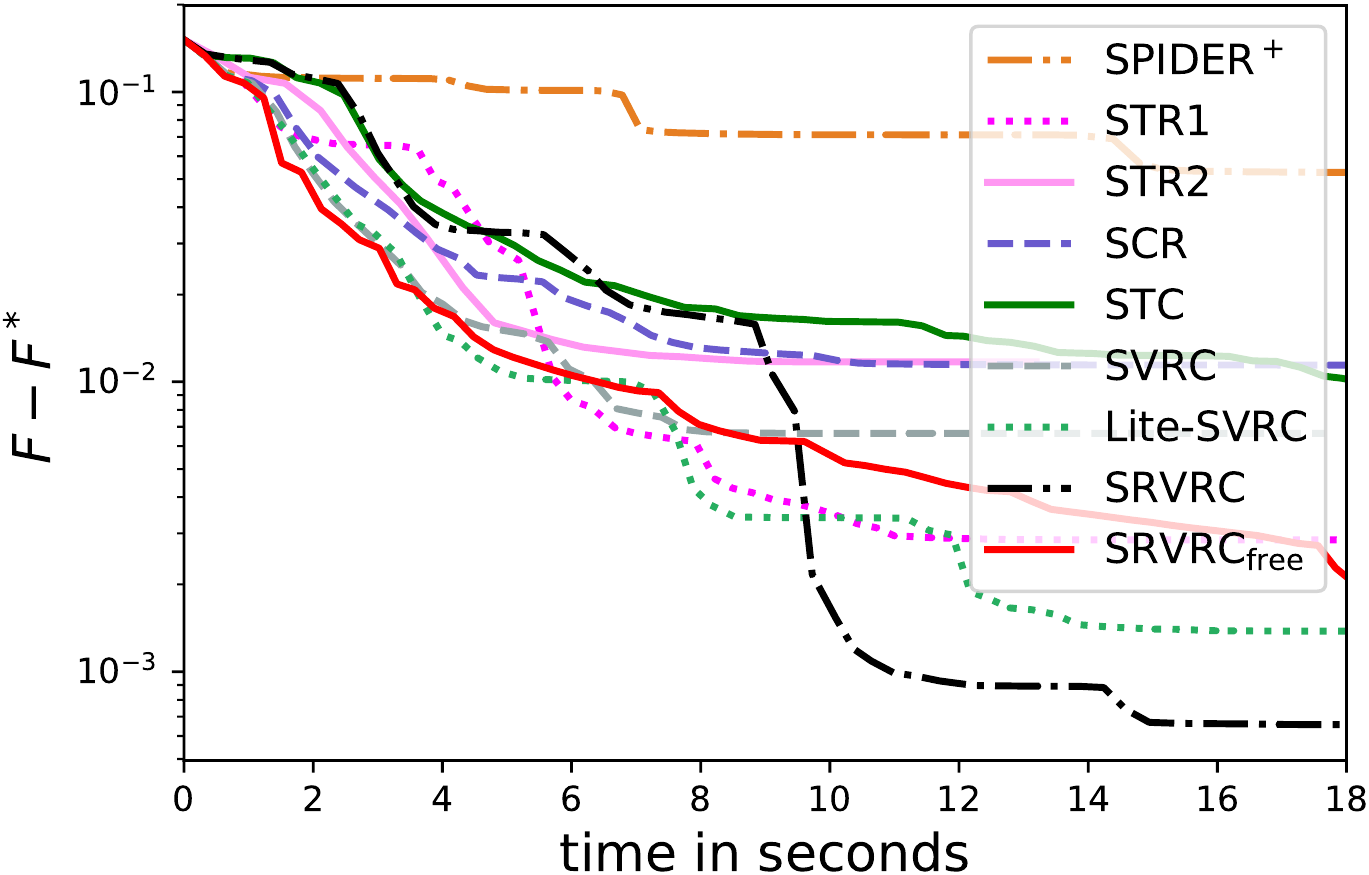}
		\label{fig:covtype}}
		\subfigure[ijcnn1]{\includegraphics[width=0.32\linewidth]{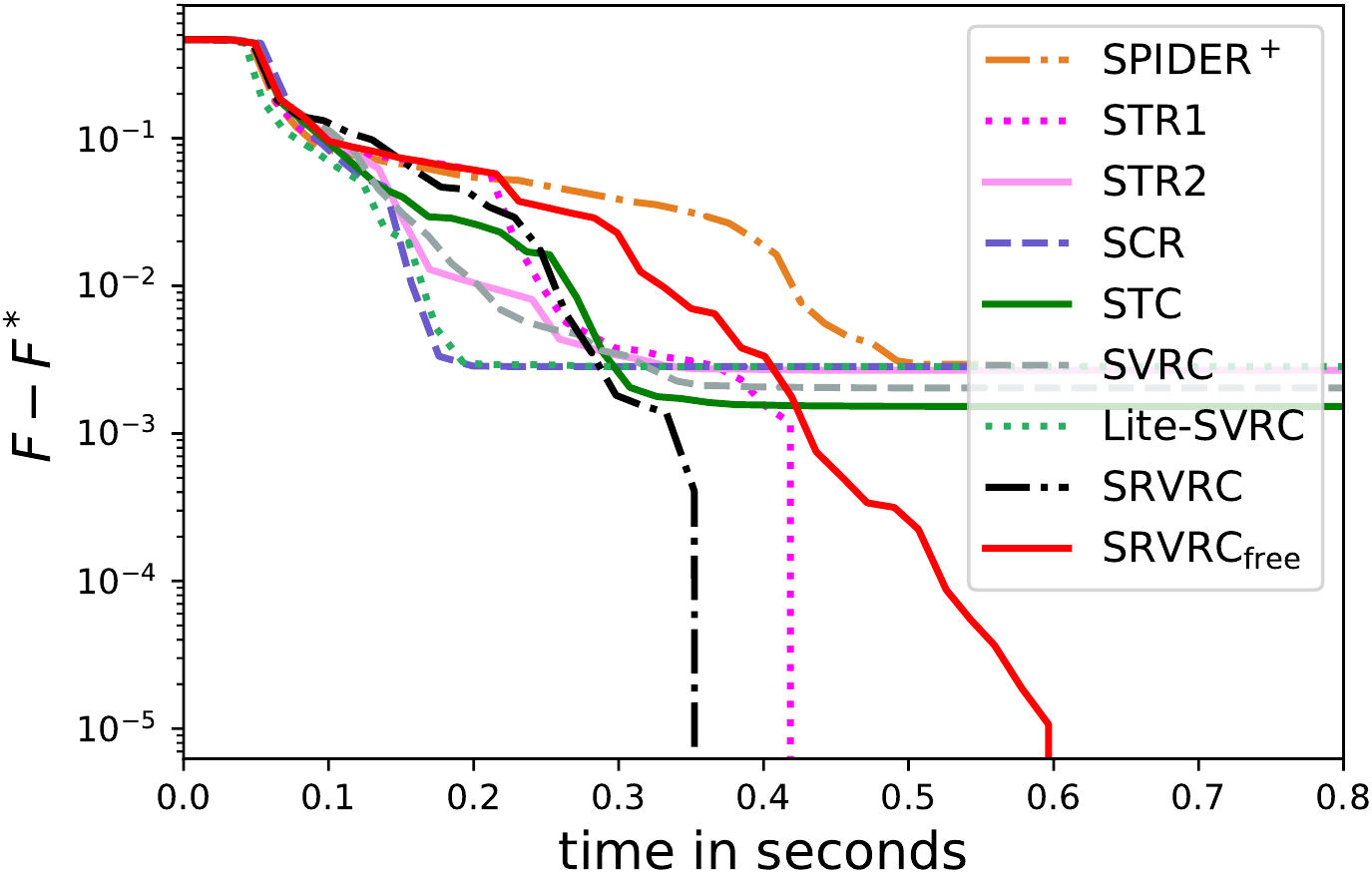}\label{fig:ijcnn1}}
		
		\subfigure[mnist]{\includegraphics[width=0.32\linewidth]{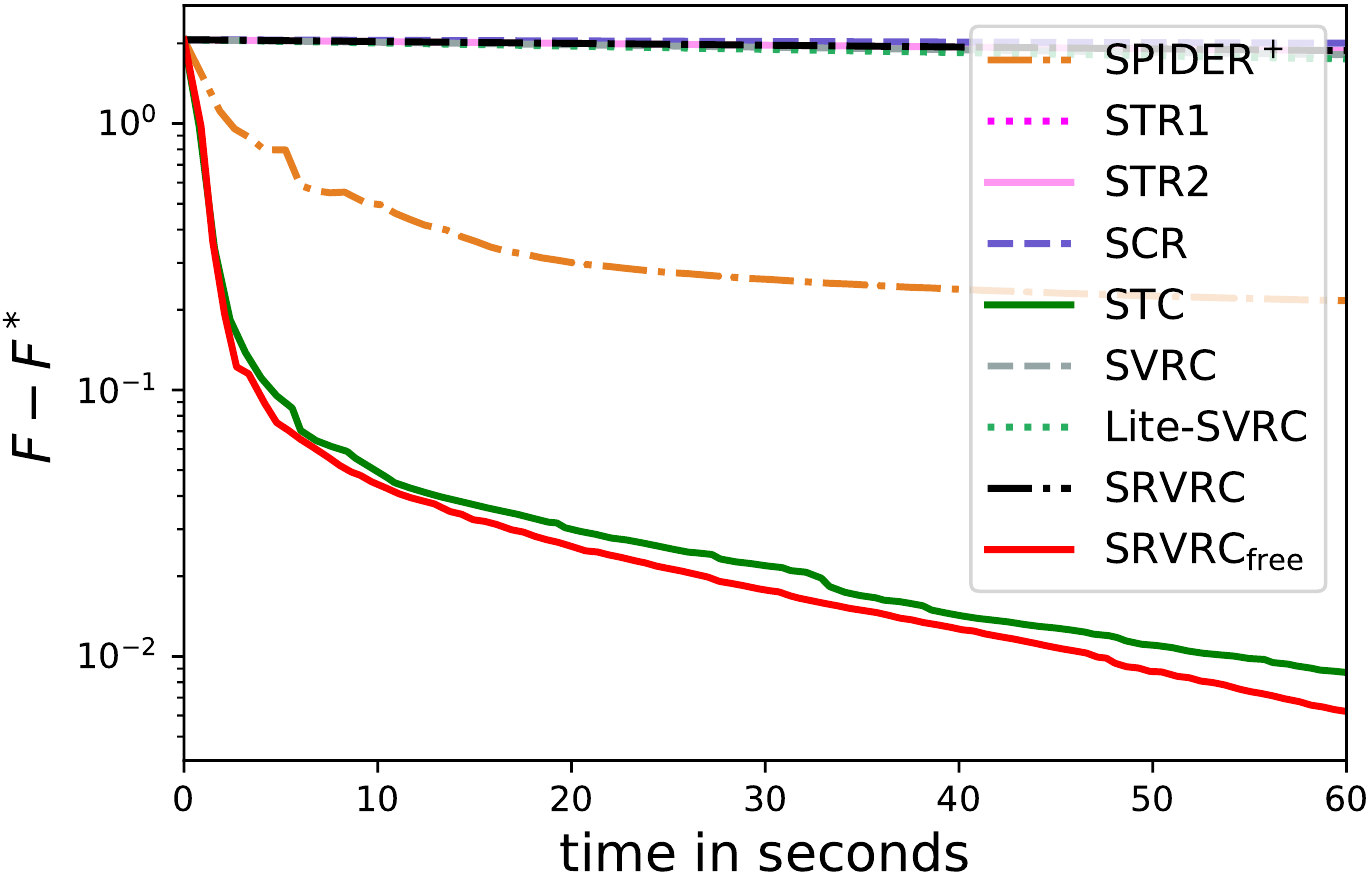}
		\label{fig:mnist}}
		\subfigure[cifar10]{\includegraphics[width=0.32\linewidth]{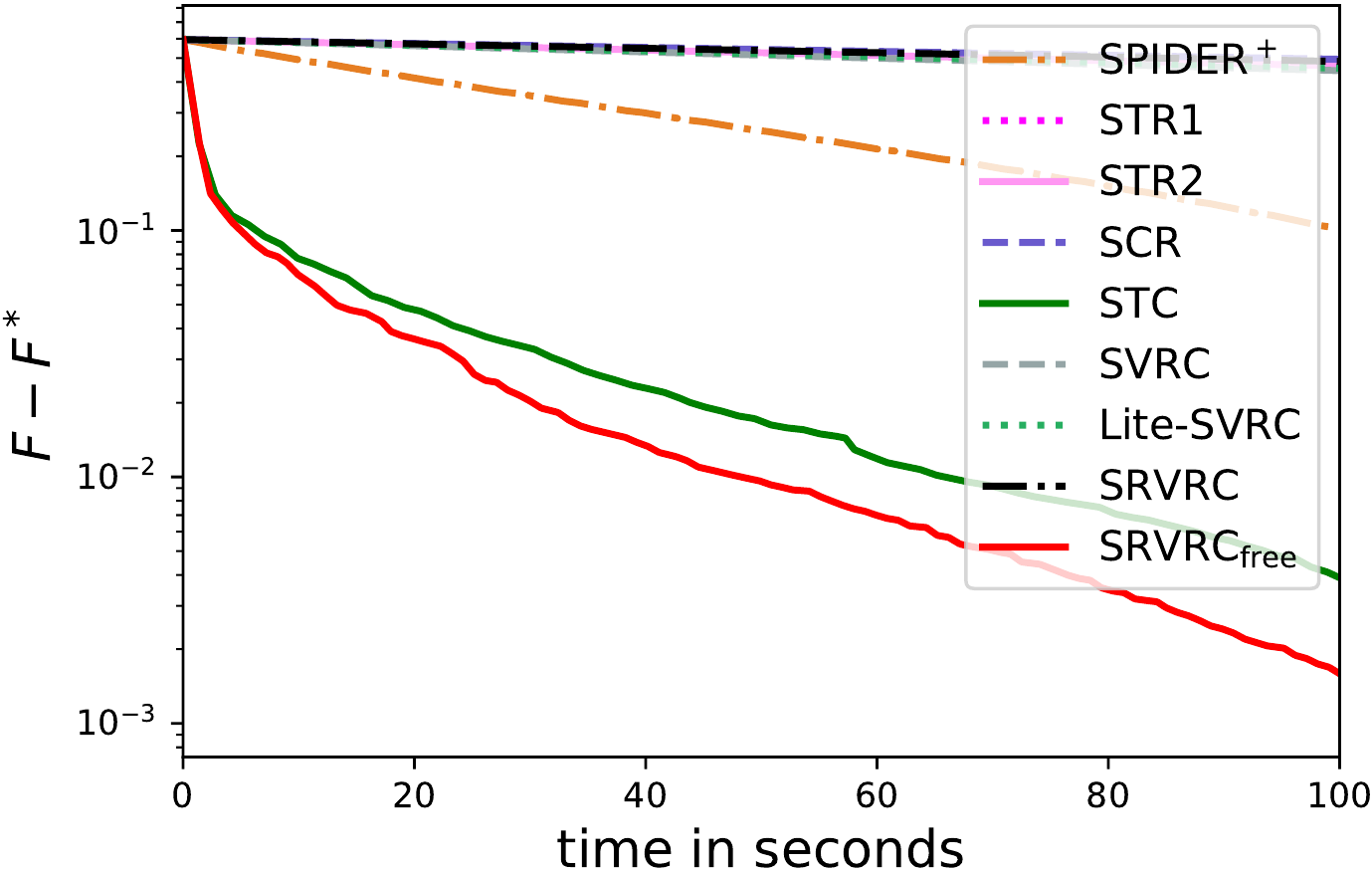}
		\label{fig:cifar10}}
		\subfigure[SVHN]{\includegraphics[width=0.32\linewidth]{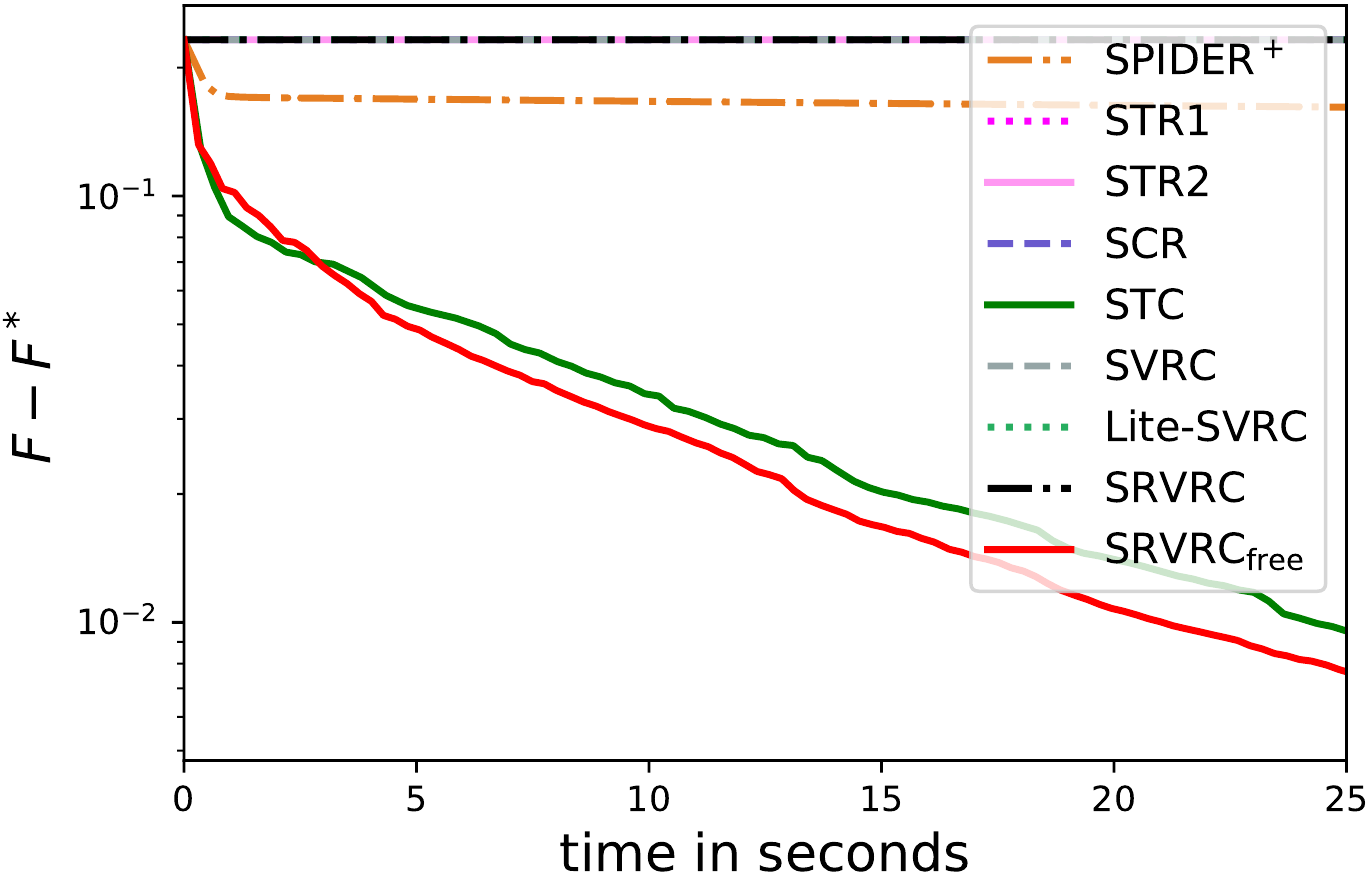}\label{fig:svhn}}
	\caption{Plots of logarithmic function value gap with respect to CPU time (in seconds) for nonconvex regularized binary logistic regression on (a) \emph{a9a} (b) \emph{ovtype} (c) \emph{ijcnn1} and for nonconvex regularized multiclass logistic regression on (d) \emph{mnist} (e) \emph{cifar10} (f) \emph{SVHN}. Best viewed in color.} \label{fig:real_nonconvex}
	\end{center}
\end{figure*}
In this section, we present numerical experiments on different nonconvex empirical risk minimization (ERM) problems and on different datasets to validate the advantage of our proposed $\nameheavy$ and $\namefree$ algorithms for finding approximate local minima.

\noindent\textbf{Baselines:} We compare our algorithms with the following algorithms:
SPIDER+ \citep{fang2018spider}, which is the local minimum finding version of SPIDER, stochastic trust region (STR1, STR2) \citep{shen2019stochastic},
subsampled cubic regularization (SCR) \citep{kohler2017sub}, stochastic cubic regularization (STC) \citep{tripuraneni2018stochastic}, stochastic variance-reduced cubic regularization (SVRC) \citep{zhou2018stochastic}, sample efficient SVRC (Lite-SVRC) \citep{zhou2018sample,wang2018sample,zhang2018adaptive}. 

\noindent\textbf{Parameter Settings and Subproblem Solver}
For each algorithm, we set the cubic penalty parameter $M_t$ adaptively based on how well the model approximates the real objective as suggested in \citep{Cartis2011Adaptive, Cartis2011Adaptive2, kohler2017sub}. For $\nameheavy$, we set $S^{(g)} = S^{(h)} = S$ for the simplicity and set gradient and Hessian batch sizes $B_t^{(g)}$ and $B_t^{(h)}$ as follows:
\begin{align}
    &B_t^{(g)} = B^{(g)}, B_t^{(h)} = B^{(h)}, &\mod(t,S) = 0,\notag \\
    &B_t^{(g)} = \lfloor B^{(g)}/S\rfloor, B_t^{(h)} = \lfloor B^{(h)}/S\rfloor, &\mod(t,S) \neq 0.\notag 
\end{align}
For $\namefree$, we set gradient batch sizes $B_t^{(g)}$ the same as $\nameheavy$ and Hessian batch sizes $B_t^{(h)} = B^{(h)}$. We tune $S$ over the grid $\{5,10,20,50\}$, $B^{(g)}$ over the grid $\{n, n/10, n/20, n/100\}$, and $B^{(h)}$ over the grid $\{50, 100, 500, 1000\}$ for the best performance. For SCR, SVRC, Lite-SVRC, and $\nameheavy$, we solve the cubic subproblem using the cubic subproblem solver discussed in \citep{Nesterov2006Cubic}. For STR1 and STR2, we solve the trust-region subproblem using the exact trust-region subproblem solver discussed in \citep{conn2000trust}. For STC and $\namefree$, we use $\namesub$ (Algorithm \ref{alg:sub} in Appendix \ref{appendix: add_apgorithm}) 
to approximately solve the cubic subproblem. All algorithms are carefully tuned for a fair comparison. 

\noindent\textbf{Datasets and Optimization Problems}
We use 6 datasets \emph{a9a}, \emph{covtype}, \emph{ijcnn1} , \emph{mnist}, \emph{cifar10} and \emph{SVHN} from \citet{chang2011libsvm} . For binary logistic regression problem with a nonconvex regularizer on \emph{a9a}, \emph{covtype}, and \emph{ijcnn1}, we are given training data $\{\xb_i, y_i\}_{i=1}^n$, where $\xb_i \in \RR^d$ and $y_i \in \{0,1\}$ are feature vector and output label corresponding to the $i$-th training example. The nonconvex penalized binary logistic regression is formulated as follows
\begin{align*}
    \min_{\wb \in \RR^d}& \frac{1}{n}\sum_{i=1}^n y_i \log \phi(\xb_i^\top\wb)+(1-y_i)\log[1-\phi(\xb_i^\top\wb)] +\lambda\sum_{i=1}^d  \frac{w_i^2}{1+w_i^2},
\end{align*}
where $\phi(x)$ is the sigmoid function and $\lambda = 10^{-3}$. 
For multiclass logistic regression problem with a nonconvex regularizer on \emph{mnist}, \emph{cifar10} and \emph{SVHN}, we are given training data $\{\xb_i, \yb_i\}_{i=1}^n$, where $\xb_i \in \RR^d$ and $\yb_i \in \RR^m$ are feature vectors and multilabels corresponding to the $i$-th data points. The nonconvex penalized multiclass logistic regression is formulated as follows
\begin{align}
    \min_{\Wb \in \RR^{m\times d}}& -\sum_{i=1}^n\frac{1}{n}\la\yb_i, \log[\text{softmax}(\Wb\xb_i)]\ra+\lambda\sum_{i=1}^m\sum_{j=1}^d{1+w_{i,j}^2},\notag
\end{align}
where $\text{softmax}(\ab)=\exp(\ab)/\sum_{i=1}^d \exp(a_i)$ is the softmax function and $\lambda = 10^{-3}$.

We plot the logarithmic function value gap with respect to CPU time in Figure \ref{fig:real_nonconvex}. From Figure \ref{fig:a9a} to \ref{fig:svhn}, we can see that for the low dimension optimization task on \emph{a9a}, \emph{covtype} and \emph{ijcnn1}, our $\nameheavy$ outperforms all the other algorithms with respect to CPU time. We can also observe that the stochastic trust region method STR1 is better than STR2, which is well-aligned with our discussion before. The SPIDER+ does not perform as well as other second-order methods, even though its stochastic gradient and Hessian complexity is comparable to second-order methods in theory. Meanwhile, we also notice that $\namefree$ always outperforms STC, which suggests that the variance reduction technique is useful. For high dimension optimization task \emph{mnist}, \emph{cifar10} and \emph{SVHN}, only SPIDER+, STC and $\namefree$ are able to make notable progress and $\namefree$ outperforms the other two. This is again consistent with our theory and discussions in Section~\ref{sec:hessfree}. Overall, our experiments clearly validate the advantage of  $\nameheavy$ and $\namefree$, and corroborate the theory of both algorithms.
\section{Conclusions and Future Work}
In this work we present two faster SVRC algorithms namely $\nameheavy$ and $\namefree$ to find approximate local minima for nonconvex finite-sum optimization problems. $\nameheavy$ outperforms existing SVRC algorithms in terms of gradient and Hessian complexities, while $\namefree$ further outperforms the best-known runtime complexity for existing CR based algorithms. Whether our algorithms have achieved the optimal complexity under the current assumptions is still an open problem, and we leave it as a future work.

\appendix

\section{Proofs in Section \ref{sec:SVRC}}\label{app_a}
 We define the filtration $\cF_t = \sigma (\xb_0,...,\xb_t)$ as the $\sigma$-algebra of $\xb_0$ to $\xb_t$. Recall that $\vb_t$ and $\Ub_t$ are the semi-stochastic gradient and Hessian respectively, $\hb_t$ is the update parameter, and $M_t$ is the cubic penalty parameter appeared in Algorithm \ref{algorithm:1} and Algorithm \ref{algorithm:2}. We denote $m_t(\hb): = \vb^\top\hb + \hb^\top\Ub_t\hb/2 + M_t\|\hb\|_2^3/6$ and $\hb_t^* = \argmin_{\hb \in \RR^d}m_t(\hb)$. In this section, we define $\delta = \xi/(2T)$ for the simplicity.

\subsection{Proof of Theorem \ref{Theorem1}}
To prove Theorem \ref{Theorem1}, we need the following lemma adapted from \citet{zhou2018stochastic}, which characterizes that $\mu(\xb_t+\hb)$ can be bounded by $\|\hb\|_2$ and the norm of difference between semi-stochastic gradient and Hessian.


\begin{lemma}\label{mu_and_h}
Suppose that $m_t(\hb):=\vb_t^\top\hb + \hb^\top\Ub_t\hb/2 + M_t\|\hb\|_2^3/6$ and $\hb_t^* = \argmin_{\hb \in \RR^d}m_t(\hb)$. If $M_{t}/\rho\geq 2$, then for any $\hb \in \RR^d$, we have
\begin{align*}
\mu(\xb_t+\hb)
& \leq
 9 \Big[M_{t}^{3}\rho^{-3/2}\|\hb\|_2^3 + M_{t}^{3/2}\rho^{-3/2} \big\|\dF(\xb_t) - \vb_t\big\|_2^{3/2}+ \rho^{-3/2} \big\|\HF(\xb_t) - \Ub_t\big\|_2^3 \notag \\
 &\qquad+ M_t^{3/2}\rho^{-3/2}\|\nabla m_t(\hb)\|_2^{3/2} + M_t^3\rho^{-3/2}\big|\|\hb\|_2 - \|\hb_t^*\|_2 \big|^3\Big].
 \end{align*}
\end{lemma}

Next lemma gives bounds on the inner products $\la\dF(\xb_t) - \vb_t,\hb\ra$ and $\la\big(\nabla^2F(\xb_t) -\Ub_t\big)\hb,\hb\ra$. 
\begin{lemma}\label{lemma:youngs}
For any $\hb \in \RR^d$, we have 
\begin{align}
\la\dF(\xb_t) - \vb_t,\hb\ra & \leq \frac{\rho}{8} \|\hb\|_2^3 + \frac{6 \|\dF(\xb_t) - \vb_t\|_2^{3/2}}{5\sqrt{\rho}},\notag\\
    \big\la\big(\nabla^2F(\xb_t) - \Ub_t\big)\hb,\hb\big\ra &\leq\frac{\rho}{8} \|\hb\|_2^3 + \frac{10}{\rho^2} \big\|\nabla^2F(\xb_t) - \Ub_t\big\|_2^3.\notag
\end{align}
\end{lemma}

We also need the  following two lemmas, which show that semi-stochastic gradient and Hessian $\vb_t$ and $\Ub_t$ estimators are good approximations to true gradient and Hessian.

\begin{lemma}\label{gradientVariance}
Suppose that $\{B_k^{(g)}\}$ satisfies \eqref{gradientVariance-1} and \eqref{gradientVariance-2}, then conditioned on $\cF_{\lfloor t/S^{(g)} \rfloor\cdot S^{(g)}}$, with probability at least $1-\delta\cdot(t - \lfloor t/S^{(g)}\rfloor\cdot S^{(g)})$, we have that for all $\lfloor t/S^{(g)} \rfloor\cdot S^{(g)} \leq k \leq t$,
\begin{align}
    \|\dF(\xb_k) - \vb_k\|_2^2 \leq \frac{\epsilon^2}{30}.
\end{align}
\end{lemma}

\begin{lemma}\label{HessianVariance}
Suppose that $\{B_k^{(h)}\}$ satisfies \eqref{hessianVariance-1} and \eqref{hessianVariance-2}, then conditioned on $\cF_{\lfloor t/S^{(h)} \rfloor\cdot S^{(h)}}$, with probability at least $1-\delta\cdot(t - \lfloor t/S^{(h)}\rfloor\cdot S^{(h)})$, we have that for all $\lfloor t/S^{(h)} \rfloor\cdot S^{(h)} \leq k \leq t$,
\begin{align}
    \|\HF(\xb_k) - \Ub_k\|_2^2 \leq \frac{\rho\epsilon}{20}.
\end{align}
\end{lemma}

Given all the above lemmas, we are ready to prove Theorem \ref{Theorem1}.
\begin{proof}[Proof of Theorem \ref{Theorem1}]

Suppose that $\nameheavy$ terminates at iteration $T^*-1$, then $\|\hb_t\|_2 > \sqrt{\epsilon/\rho}$ for all $0 \leq t \leq T^*-1$. We have
\begin{align}
F(\xb_{t+1}) &\leq F(\xb_t) + \la \dF(\xb_t), \hb_t\ra + \frac{1}{2}\la \hb_t, \HF(\xb_t)\hb_t\ra + \frac{\rho}{6}\|\hb_t\|_2^3\notag \\
&  = F(\xb_t) + m_t(\hb_t) + \frac{\rho - M_{t}}{6}\|\hb_t\|_2^3 
+\la \hb_t , \dF(\xb_t) - \vb_t\ra 
+ \frac{1}{2}\la \hb_t, (\HF(\xb_t) - \Ub_t)\hb_t\ra\notag \\
& \leq F(\xb_t) - \frac{\rho}{2}\|\hb_t\|_2^3 + \frac{\rho }{4}\|\hb_t\|_2^3 +\frac{6\|\dF(\xb_t) - \vb_t\|_2^{3/2}}{5\sqrt{\rho}}+ \frac{10}{\rho^2}\|\HF(\xb_t) - \Ub_t\|_2^3\notag \\
& = 
F(\xb_t) - \frac{\rho}{4}\|\hb_t\|_2^3 + \frac{6\|\dF(\xb_t) - \vb_t\|_2^{3/2}}{5\sqrt{\rho}}+\frac{10}{\rho^2}\|\HF(\xb_t) - \Ub_t\|_2^3,\label{Theorem1_0}
\end{align}
where the second inequality holds due to the fact that $m_t(\hb_t) \leq m_t(\zero) = 0$, $ M_t = 4\rho$ and Lemma \ref{lemma:youngs}.
By Lemmas \ref{gradientVariance} and \ref{HessianVariance}, with probability at least $1-2T\delta$, for all $0 \leq t \leq T-1$, we have that 
\begin{align}
    \|\dF(\xb_t) - \vb_t\|_2^{3/2} \leq \frac{\epsilon^{3/2}}{12},\quad \|\HF(\xb_t) - \Ub_t\|_2^3 \leq \frac{(\rho\epsilon)^{3/2}}{80}\label{Theorem1_1}
\end{align}
for all $0 \leq t \leq T-1$. Substituting \eqref{Theorem1_1} into \eqref{Theorem1_0}, we have
\begin{align}
    F(\xb_{t+1}) \leq F(\xb_{t}) - \frac{\rho}{4}\|\hb_t\|_2^3 + \frac{9\rho^{-1/2}\epsilon^{3/2}}{40}.\label{theorem_11}
\end{align}
Telescoping \eqref{theorem_11} from $t= 0,\dots,T^*-1$, we have
\begin{align}
    \Delta_F\geq F(\xb_0) - F(\xb_{T^*}) \geq \rho\cdot T^*\cdot(\epsilon/\rho)^{3/2}/4-9/40\cdot\rho^{-1/2}\epsilon^{3/2}\cdot T^*  = \rho^{-1/2}\epsilon^{3/2}\cdot T^*/40.\label{Theorem1_1.11}
\end{align}
Recall that we have $T \geq 40 \Delta_F\sqrt{\rho}/\epsilon^{3/2}$ from the condition of Theorem \ref{Theorem1}, then by \eqref{Theorem1_1.11}, we have $T^* \leq T$. 
Thus, we have $\|\hb_{T^*-1}\|_2 \leq \sqrt{\epsilon/\rho}$. Denote $\tilde T =T^* - 1$, then we have
\begin{align}
    \mu(\xb_{\tilde T +1}) &= \mu(\xb_{\tilde T}+\hb_{\tilde T})\notag \\
    & \leq 9 \Big[M_{\tilde T}^{3}\rho^{-3/2}\|\hb_{\tilde T}\|_2^3 +M_{\tilde T}^{3/2}\rho^{-3/2} \big\|\dF(\xb_{\tilde T}^s) - \vb_{\tilde T}\big\|_2^{3/2}+ \rho^{-3/2} \big\|\HF(\xb_{\tilde T}) - \Ub_{\tilde T}\big\|_2^3 \Big]\notag \\
    &\leq 
 600\epsilon^{3/2},\notag
\end{align}
where the first inequality holds due to Lemma \ref{mu_and_h} with $\nabla m_{\tilde T}(\hb_{\tilde T}) = 0$ and $\|\hb_{\tilde T}\|_2 = \|\hb_{\tilde T}^*\|_2$. This completes our proof.
\end{proof}
\subsection{Proof of Corollary \ref{corollary_1}}
\begin{proof}[Proof of Corollary \ref{corollary_1}]
Suppose that $\nameheavy$ terminates at $T^*-1 \leq T-1$ iteration. Telescoping \eqref{theorem_11} from $t = 0$ to $T^*-1$, we have
\begin{align}
    \Delta_F \geq F(\xb_0) - F(\xb_{T^*}) \geq \rho\sum_{t=0}^{T^*-1}\|\hb_t\|_2^3/4 - 9\rho^{-1/2}\epsilon^{3/2}/40\cdot T = \rho\sum_{t=0}^{T^*-1}\|\hb_t\|_2^3/4 -9\cdot \Delta_F,\label{eq:coro_-1}
\end{align}
where the last inequality holds since $T$ is set to be $40\Delta_F\sqrt{\rho}/\epsilon^{3/2}$ as the conditions of Corollary \ref{corollary_1} suggests.
\eqref{eq:coro_-1} implies that $\sum_{t=0}^{T^*-1}\|\hb_t\|_2^3 \leq 40\Delta_F/\rho$. Thus, we have
\begin{align}
    \sum_{t=0}^{T^*-1}\|\hb_t\|_2^2 \leq (T^*)^{1/3}\bigg(\sum_{t=0}^{T^*-1}\|\hb_t\|_2^3\bigg)^{2/3} \leq \bigg(\frac{40\Delta_F\rho^{1/2}}{\epsilon^{3/2}}\bigg)^{1/3}\cdot \bigg(\frac{40\Delta_F}{\rho}\bigg)^{2/3} = \frac{40\Delta_F}{\rho^{1/2}\epsilon^{1/2}},\label{coro_0}
\end{align}
where the first inequality holds due to H\"{o}lder's inequality inequality, and the second inequality is due to $T^* \leq T =40\Delta_F\sqrt{\rho}/\epsilon^{3/2} $.
We first consider the total gradient sample complexity $\sum_{t=0}^{T^*-1}B_t^{(g)}$, which can be bounded as
\begin{align}
    &\sum_{t=0}^{T^*-1}B_t^{(g)} \notag \\
    &= \sum_{\mod(t,S^{(g)}) = 0}B_t^{(g)} + \sum_{\mod(t,S^{(g)}) \neq 0}B_t^{(g)}\notag \\
    & = \sum_{\mod(t,S^{(g)}) = 0}\min\bigg\{n, 1440\frac{\upp^2\log^2(d/\delta)}{\epsilon^2}\bigg\} + \sum_{\mod(t,S^{(g)}) \neq 0}\min\bigg\{n, 1440L^2\log^2(d/\delta)\frac{S^{(g)}\|\hb_{t-1}\|_2^2}{\epsilon^2}\bigg\}\notag \\
    & \leq 
    C_1\bigg[n \land \frac{\upp^2}{\epsilon^2}+\frac{T^*}{S^{(g)}}\bigg(n \land \frac{\upp^2}{\epsilon^2}\bigg) + \bigg(\frac{L^2S^{(g)}}{\epsilon^2}\sum_{t=0}^{T^*-1}\|\hb_t\|_2^2\bigg)\land nT^*\bigg]\notag \\
    & \leq 
    40C_1\bigg[n \land \frac{\upp^2}{\epsilon^2}+\frac{\Delta_F \rho^{1/2}}{\epsilon^{3/2}S^{(g)}}\bigg(n \land \frac{\upp^2}{\epsilon^2}\bigg) + \bigg(\frac{\Delta_FL^2S^{(g)}}{\rho^{1/2}\epsilon^{5/2}}\bigg)\land\frac{n\Delta_F\rho^{1/2}}{\epsilon^{3/2}}\bigg]\notag \\
    & = \tilde O\bigg(n\land \frac{\upp^2}{\epsilon^2} + \frac{\Delta_F}{\epsilon^{3/2}}\bigg[\sqrt{\rho}n \land \frac{L \sqrt{n} }{\sqrt{\epsilon}} \land \frac{L M }{\epsilon^{3/2}}\bigg]\bigg),\notag
\end{align}
where $C_1 = 1440\log^2(d/\delta)$, the second inequality holds due to \eqref{coro_0}, and the last equality holds due to the choice of $S^{(g)} =  \sqrt{\rho\epsilon}/L\cdot\sqrt{n \land M^2/\epsilon^2}$. 
We then consider the total Hessian sample complexity $\sum_{t=0}^{T^*-1}B_t^{(h)}$, which can be bounded as
\begin{align}
    &\sum_{t=0}^{T^*-1}B_t^{(h)} \notag \\
    &= \sum_{\mod(t,S^{(h)}) = 0}B_t^{(h)} + \sum_{\mod(t,S^{(h)}) \neq 0}B_t^{(h)}\notag \\
    & = \sum_{\mod(t,S^{(h)}) = 0}\min\bigg\{n, 800\frac{L^2\log^2(d/\delta)}{\rho\epsilon}\bigg\}
    + \sum_{\mod(t,S^{(h)}) \neq 0}\min\bigg\{n, 800\rho\log^2(d/\delta)\frac{S^{(h)}\|\hb_{t-1}\|_2^2}{\epsilon}\bigg\}\notag \\
    & \leq 
    C_2\bigg[n \land \frac{ L^2}{\rho \epsilon}+\frac{T^*}{S^{(h)}}\bigg(n \land \frac{ L^2}{\rho \epsilon}\bigg) + \frac{\rho S^{(h)}}{\epsilon}\sum_{t=0}^{T^*-1}\|\hb_t\|_2^2\bigg]\notag \\
     &\leq 40C_2\bigg[n \land \frac{ L^2}{\rho \epsilon}
     +\frac{\Delta_F \rho^{1/2}}{\epsilon^{3/2}S^{(h)}}\bigg(n \land \frac{ L^2}{\rho \epsilon}\bigg) + \frac{\Delta_F\rho^{1/2}S^{(h)}}{\epsilon^{3/2}}
     \bigg]\notag \\
     & = 
    \tilde O\bigg[n \land \frac{ L^2}{\rho \epsilon}
     +\frac{\Delta_F \rho^{1/2}}{\epsilon^{3/2}}\sqrt{n \land \frac{ L^2}{\rho \epsilon}}
     \bigg],\notag
\end{align}
where $C_2 = 800\log^2(d/\delta)$, the second inequality holds due to \eqref{coro_0}, and the last equality holds due to the choice of $S^{(h)} = \sqrt{n \land L/(\rho\epsilon)}$.
\end{proof}

\section{Proofs in Section \ref{sec:hessfree}}\label{app:b}
In this section, we denote $\delta = \xi/(3T)$ for simplicity.
\subsection{Proof of Theorem \ref{thm:2}}

We need the following two lemmas, which bound the variance of semi-stochastic gradient and Hessian estimators.
\begin{lemma}\label{gradientVariance2}
Suppose that $\{B_k^{(g)}\}$ satisfies \eqref{gradientVariance-3} and \eqref{gradientVariance-4}, then conditioned on $\cF_{\lfloor t/S \rfloor\cdot S}$, with probability at least $1-\delta\cdot(t - \lfloor t/S\rfloor\cdot S)$, we have that for all $\lfloor t/S \rfloor\cdot S \leq k \leq t$,
\begin{align}
    \|\dF(\xb_k) - \vb_k\|_2^2 \leq \frac{\epsilon^2}{55}.\notag
\end{align}
\end{lemma}
\begin{proof}[Proof of Lemma \ref{gradientVariance2}]
The proof is very similar to that of Lemma \ref{gradientVariance}, hence we omit it.
\end{proof}

\begin{lemma}\label{HessianVariance2}
Suppose that $\{B_k^{(h)}\}$ satisfies \eqref{HessianVariance-3}, then conditioned on $\cF_k$, with probability at least $1-\delta$, we have that 
\begin{align}
    \|\HF(\xb_k) - \Ub_k\|_2^2 \leq \frac{\rho\epsilon}{30}.\notag
\end{align}
\end{lemma}
\begin{proof}[Proof of Lemma \ref{HessianVariance2}]
The proof is very similar to that of Lemma \ref{HessianVariance}, hence we omit it.
\end{proof}

We have the following lemma to guarantee that by Algorithm \ref{alg:sub} $\namesub$, the output $\hb_t$ satisfies that sufficient decrease of function value will be made and the total number of iterations is bounded by $T'$.
\begin{lemma}\label{sufficientde}
For any $t\geq 0$, suppose that $\|\hb_t^*\|_2 \geq \sqrt{\epsilon/\rho}$ or $\|\vb_t\|_2 \geq \max\{M_t\epsilon/(2\rho), \sqrt{LM_t/2}(\epsilon/\rho)^{3/4}\} $. We set $\eta  = 1/(16L)$. Then for $\epsilon < 16L^2\rho/M_t^2$, with probability at least $1-\delta$, $\namesub (\Ub_t, \vb_t, M_t, \eta, \sqrt{\epsilon/\rho}, 0.5,\delta )$ will return $\hb_t$ satisfying $m_t(\hb_t) \leq -M_t\rho^{-3/2}\epsilon^{3/2}/24$. within
\begin{align}
    T' =C_S \frac{L}{M_t\sqrt{\epsilon/\rho}}\notag
\end{align}
iterations, where $C_S>0$ is a constant. 
\end{lemma}
We have the following lemma which provides the guarantee for the dynamic of gradient steps in $\namefinal$.

\begin{lemma}\label{carmon_1}\citep{Carmon2016Gradient}
For $\bbb, \Ab, \tau$, suppose that $\|\Ab\|_2 \leq L$. 
We denote that $g(\hb) = \bbb^\top\hb + \hb^\top\Ab\hb/2 + \tau/6\cdot\|\hb\|_2^3$, $\sbb = \argmin_{\hb \in \RR^d}g(\hb)$, and let $R$ be 
\begin{align}
    R = \frac{L}{2\tau}+\sqrt{\bigg(\frac{L}{2\tau}\bigg)^2 + \frac{\|\bbb\|_2}{\tau}}.\notag
\end{align}
Then for $\namefinal$, suppose that $\eta<(4(L+\tau R))^{-1}$, then each iterate $\Delta$ in $\namefinal$ satisfies that $\|\Delta\|_2 \leq \|\sbb\|_2$, and $g(\hb)$ is $(L+2\tau R)$-smooth.
\end{lemma}

With these lemmas, we begin our proof of Theorem \ref{thm:2}.
\begin{proof}[Proof of Theorem \ref{thm:2}]
Suppose that $\namefree$ terminates at iteration $T^*-1$. Then $T^* \leq T$. We first claim that $T^*<T$. Otherwise, suppose $T^* = T$, then we have that for all $0 \leq t <T^*$,
\begin{align}
    F(\xb_{t+1}) &\leq F(\xb_t) + \la \dF(\xb_t), \hb_t\ra + \frac{1}{2}\la \hb_t, \HF(\xb_t)\hb_t\ra + \frac{\rho}{6}\|\hb_t\|_2^3\notag \\
&  = F(\xb_t) + m_t(\hb_t) + \frac{\rho - M_{t}}{6}\|\hb_t\|_2^3 
+\la \hb_t , \dF(\xb_t) - \vb_t\ra 
+ \frac{1}{2}\la \hb_t, (\HF(\xb_t) - \Ub_t)\hb_t\ra\notag \\
& \leq F(\xb_t) - \frac{\rho}{4}\|\hb_t\|_2^3
+m_t(\hb_t)
+\frac{6\|\dF(\xb_t) - \vb_t\|_2^{3/2}}{5\sqrt{\rho}}+ \frac{10}{\rho^2}\|\HF(\xb_t) - \Ub_t\|_2^3,\label{Theorem2_0}
\end{align}
where the second inequality holds due to $M_t =4\rho$ and Lemma \ref{lemma:youngs}.
By Lemma \ref{sufficientde} and union bound, we know that with probability at least $1-T\delta$, we have
\begin{align}
    m_t(\hb_t) \leq -M_t\rho^{-3/2}\epsilon^{3/2}/24= -\rho^{-1/2}\epsilon^{3/2}/6,\label{Theorem2_1}
\end{align}
where we use the fact that $M_t = 4\rho$. By Lemmas \ref{gradientVariance2} and \ref{HessianVariance2}, we know that with probability at least $1-2T\delta$, for all $0 \leq t\leq T^*-1$, we have
\begin{align}
    \|\dF(\xb_t) - \vb_t\|_2^{3/2} \leq \epsilon^{3/2}/20,\quad \|\HF(\xb_t) - \Ub_t\|_2^3 \leq (\rho\epsilon)^{3/2}/160.\label{Theorem2_2}
\end{align}
Substituting \eqref{Theorem2_1} and \eqref{Theorem2_2} into \eqref{Theorem2_0}, we have
\begin{align}
    F(\xb_{t+1}) - F(\xb_t) \leq -\rho^{-1/2}\epsilon^{3/2}/6 - \rho\|\hb_t\|_2^3/4 + \rho^{-1/2}\epsilon^{3/2}/8 \leq -\rho\|\hb_t\|_2^3/4 -\rho^{-1/2}\epsilon^{3/2}/24.\label{Theorem2_3}
\end{align}
Telescoping \eqref{Theorem2_3} from $t = 0$ to $T^*-1$, we have
\begin{align}
    \Delta_F \geq F(\xb_0) - F(\xb_{T^*}) \geq \rho \sum_{t=0}^{T^*-1}\|\hb_t\|_2^3/4 + \rho^{-1/2}\epsilon^{3/2}\cdot T^*/24 > \rho \sum_{t=0}^{T^*-1}\|\hb_t\|_2^3/4 + \Delta_F,\label{Theorem2_4}
\end{align}
where the last inequality holds since we assume $T^* = T \geq 25\Delta_F\rho^{1/2}\epsilon^{-3/2}$ from the condition of Theorem \ref{thm:2}.  \eqref{Theorem2_4} leads to a contradiction, thus we have $T^* < T$. 
Therefore, by union bound, with probability at least $1-3T\delta$, $\namefinal$ is executed by $\namefree$ at $T^*-1$ iteration. We have that $\|\vb_{T^*-1}\|_2 < \max\{M_{T^*-1}\epsilon/(2\rho), \sqrt{LM_{T^*-1}/2}(\epsilon/\rho)^{3/4}\} $ and $\|\hb_{T^*-1}^*\|_2 < \sqrt{\epsilon/\rho}$ by Lemma \ref{sufficientde}.

The only thing left is to check that we indeed find a second-order stationary point, $\xb_{T^*}$, by $\namefinal$. We first need to check that the choice of $\eta = 1/(16L)$ satisfies that $1/\eta>4(L+M_t R)$ by Lemma \ref{carmon_1}, where 
\begin{align}
    R =  \frac{L}{2M_{T^*-1}}+\sqrt{\bigg(\frac{L}{2M_{T^*-1}}\bigg)^2 + \frac{\|\vb_{T^*-1}\|_2}{M_{T^*-1}}},\notag
\end{align} 
We can check that with the assumption that $\|\vb_{T^*-1}\|_2 < \max\{M_{T^*-1}\epsilon/(2\rho), \sqrt{LM_{T^*-1}/2}(\epsilon/\rho)^{3/4}\} $, if $\epsilon < 4L^2\rho/M_{T^*-1}^2$, then $1/\eta>4(L+M_{T^*-1} R)$ holds. 

For simplicity, we denote $\tilde T = T^* - 1$. Then we have
\begin{align}
    \mu(\xb_{\tilde T} + \hb_{\tilde T})
    & \leq 9 \Big[M_{\tilde T}^{3}\rho^{-3/2}\|\hb_{\tilde T}\|_2^3 + M_{\tilde T}^{3/2}\rho^{-3/2} \big\|\dF(\xb_{\tilde T}) - \vb_{\tilde T}\big\|_2^{3/2}+ \rho^{-3/2} \big\|\HF(\xb_{\tilde T}) - \Ub_{\tilde T}\big\|_2^3 \notag \\
 &\qquad+ M_{\tilde T}^{3/2}\rho^{-3/2}\|\nabla m_{\tilde T}(\hb_{\tilde T})\|_2^{3/2} + M_{\tilde T}^3\rho^{-3/2}\big|\|\hb_{\tilde T}\|_2 - \|\hb_{\tilde T}^*\|_2 \big|^3\Big]\notag \\
 & \leq 9 \Big[2M_{\tilde T}^{3}\rho^{-3/2}\|\hb_{\tilde T}^*\|_2^3 + M_{\tilde T}^{3/2}\rho^{-3/2} \big\|\dF(\xb_{\tilde T}) - \vb_{\tilde T}\big\|_2^{3/2}+ \rho^{-3/2} \big\|\HF(\xb_{\tilde T}) - \Ub_{\tilde T}\big\|_2^3 \notag \\
 &\qquad+ M_{\tilde T}^{3/2}\rho^{-3/2}\|\nabla m_{\tilde T}(\hb_{\tilde T})\|_2^{3/2} \Big]\notag \\
 & \leq 1300 \epsilon^{3/2},\notag
\end{align}
where the first inequality holds due to Lemma \ref{mu_and_h}, the second inequality holds due to the fact that $\|\hb_{\tilde T}\|_2 \leq \|\hb_{\tilde T}^*\|_2 $ from Lemma \ref{carmon_1}, the last inequality holds due to the facts that $\|\nabla m_{\tilde T}(\hb_{\tilde T})\|_2 \leq \epsilon$ from $\namefinal$ and $\|\hb_{\tilde T}^*\|_2 \leq \sqrt{\epsilon/\rho}$ by Lemma \ref{sufficientde}.
\end{proof}

\subsection{Proof of Corollary \ref{free_total_count}}
We have the following lemma to bound the total number of iterations $T''$ of Algorithm \ref{alg:finalsub} $\namefinal$.
\begin{lemma}\label{finalitera}
If $\epsilon < 4L^2\rho/M_t^2$, then $\namefinal$ will terminate within $T'' = C_FL/\sqrt{\rho\epsilon}$ iterations, where $C_F>0$ is a constant. 
\end{lemma}
\begin{proof}[Proof of Corollary \ref{free_total_count}]
We have that
\begin{align}
    \sum_{t=0}^{T^*-1}\|\hb_t\|_2^2 \leq (T^*)^{1/3}\bigg( \sum_{t=0}^{T^*-1}\|\hb_t\|_2^3\bigg)^{2/3} \leq \bigg(\frac{25\Delta_F\rho^{1/2}}{\epsilon^{3/2}}\bigg)^{1/3}\cdot\bigg(\frac{4\Delta_F}{\rho}\bigg)^{2/3} \leq \frac{\Delta_F}{8\rho^{1/2}\epsilon^{1/2}} ,\label{coro2_0}
\end{align}
where the first inequality holds due to H\"{o}lder's inequality, the second inequality holds due to the facts that $T^* \leq T =25\Delta_F\rho^{1/2}/\epsilon^{3/2}$ and $\Delta_F \geq \rho\sum_{t=0}^{T^*-1}\|\hb_t\|_2^3/4$ by \eqref{Theorem2_4}. We first consider the total stochastic gradient computations, $\sum_{t=0}^{T^*-1}B_t^{(g)}$, which can be bounded as
\begin{align}
    &\sum_{t=0}^{T^*-1}B_t^{(g)} \notag \\
    &= \sum_{\mod(t,S^{(g)}) = 0}B_t^{(g)} + \sum_{\mod(t,S^{(g)}) \neq 0}B_t^{(g)}\notag \\
    & = \sum_{\mod(t,S^{(g)}) = 0}\min\bigg\{n, 2640\frac{\upp^2\log^2(d/\delta)}{\epsilon^2}\bigg\} + \sum_{\mod(t,S^{(g)}) \neq 0}\min\bigg\{n, 2640L^2\log^2(d/\delta)\frac{S^{(g)}\|\hb_{t-1}\|_2^2}{\epsilon^2}\bigg\}\notag \\
    & \leq 
    C_1\bigg[n \land \frac{\upp^2}{\epsilon^2}+\frac{T^*}{S^{(g)}}\bigg(n \land \frac{\upp^2}{\epsilon^2}\bigg) + \bigg(\frac{L^2S^{(g)}}{\epsilon^2}\sum_{t=0}^{T^*-1}\|\hb_t\|_2^2\bigg)\land nT^*\bigg]\notag \\
    & \leq 
    8C_1\bigg[n \land \frac{\upp^2}{\epsilon^2}+\frac{\Delta_F \rho^{1/2}}{\epsilon^{3/2}S^{(g)}}\bigg(n \land \frac{\upp^2}{\epsilon^2}\bigg) + \bigg(\frac{\Delta_FL^2S^{(g)}}{\rho^{1/2}\epsilon^{5/2}}\bigg)\land\frac{n\Delta_F\rho^{1/2}}{\epsilon^{3/2}}\bigg]\notag \\
    & = 8C_1\bigg[n \land \frac{\upp^2}{\epsilon^2}+\frac{\Delta_F \rho^{1/2}}{\epsilon^{3/2}}\bigg(\frac{1}{S^{(g)}}\bigg(n \land \frac{\upp^2}{\epsilon^2}\bigg) + \bigg(\frac{L^2S^{(g)}}{\rho\epsilon}\bigg)\land n\bigg)\bigg]\notag\\
    & = 8C_1\bigg[
    n \land \frac{\upp^2}{\epsilon^2}+\frac{\Delta_F \rho^{1/2}}{\epsilon^{3/2}}\bigg(n\land \frac{L\sqrt{n}}{\sqrt{\rho\epsilon}} \land \frac{LM}{\rho^{1/2}\epsilon^{3/2}}\bigg)
    \bigg],\label{count_0}
\end{align}
where $C_1 = 2640\log^2(d/\delta)$, the second inequality holds due to \eqref{coro2_0}, the last equality holds due to the fact $S^{(g)} = \sqrt{\rho\epsilon}/L\cdot \sqrt{n \land M^2/\epsilon^2}$. We now consider the total amount of Hessian-vector product computations $\cT$, which includes $\cT_1$ from $\namesub$ and $\cT_2$ from $\namefinal$. By Lemma \ref{sufficientde}, we know that at $k$-th iteration of $\namefree$, $\namesub$ has $T'$ iterations, which needs $B_k^{(h)}$ Hessian-vector product computations each time. Thus, we have
\begin{align}
    \cT_1 &= \sum_{k=0}^{T^*-1}T'\cdot B_k^{(h)}\nonumber\\
    &\leq C_2\bigg(T\cdot T'\cdot \bigg[n\land \frac{L^2}{\rho\epsilon}\bigg]\bigg)\nonumber\\
    &\leq  25C_2\bigg( T'\frac{\Delta_F\rho^{1/2}}{\epsilon^{3/2}}\bigg[n\land \frac{L^2}{\rho\epsilon}\bigg] \bigg) \nonumber\\
    &\leq 7 C_2C_S\bigg( \frac{L\Delta_F}{\epsilon^2}\cdot \bigg[n\land \frac{L^2}{\rho\epsilon}\bigg] \bigg),\label{count_1}
\end{align}
where $C_2 = 1200\log^2(d/\delta)$, the first inequality holds due to the fact that $B_k^{(h)} = C_2 n \land (L^2/\rho\epsilon)$, the second inequality holds due to the fact that $T = 25\Delta_F \rho^{1/2}/\epsilon^{3/2}$, the last inequality holds due to the fact that $T'  = C_S L/M_t\cdot \sqrt{\rho/\epsilon} = C_S L/(4\sqrt{\rho\epsilon})$.
For $\cT_2$, we have
\begin{align}
    \cT_2  = B_{T^*-1}^{(h)}\cdot T''  \leq C_2T''\bigg[n\land \frac{L^2}{\rho\epsilon}\bigg]  \leq C_2C_F\bigg(\frac{L}{\sqrt{\rho\epsilon}}\cdot \bigg[n\land \frac{L^2}{\rho\epsilon}\bigg] \bigg),\label{count_2}
\end{align}
where the first inequality holds due to the fact that $B_{T^*-1}^{(h)} = C_2 n \land (L^2/\rho\epsilon)$, the second inequality holds due to the fact that 
$T'' = C_FL/\sqrt{\rho\epsilon}$ by Lemma \ref{finalitera}. Since at each iteration we need $B_{T^*-1}^{(h)}$ Hessian-vector computations.

Combining \eqref{count_0}, \eqref{count_1} and \eqref{count_2}, we know that the total stochastic gradient and Hessian-vector product computations are bounded as
\begin{align}
    & \sum_{t=0}^{T^*-1}B_t^{(g)}  + \cT_1 + \cT_2\notag \\
    & = \tilde O\bigg[
    n \land \frac{\upp^2}{\epsilon^2}+\frac{\Delta_F \rho^{1/2}}{\epsilon^{3/2}}\bigg(n\land \frac{L\sqrt{n}}{\sqrt{\rho\epsilon}} \land \frac{LM}{\rho^{1/2}\epsilon^{3/2}}\bigg) + \bigg(\frac{L\Delta_F}{\epsilon^2} + \frac{L}{\sqrt{\rho\epsilon}}\bigg)\cdot \bigg(n\land \frac{L^2}{\rho\epsilon}\bigg) 
    \bigg].\label{free_total_count_4}
\end{align}
\end{proof}

\section{Proofs of Technical Lemmas in Appendix \ref{app_a}}\label{app_c}

\subsection{Proof of Lemma \ref{mu_and_h}}
We have the following lemmas from \citet{zhou2018stochastic}
\begin{lemma}\label{lemma:rrf1}\citep{zhou2018stochastic}
If $M_t \geq 2\rho$, then we have
\begin{align}
    \|\nabla F(\xb_t + \hb)\|_2 \leq M_t\|\hb\|_2^2 +\big\|\dF(\xb_t) - \vb_t\big\|_2 +  \frac{1}{M_t}\big\|\HF(\xb_t) - \Ub_t\big\|_2^2 + \|\nabla m_t(\hb)\|_2.\notag
\end{align}
\end{lemma}
\begin{lemma}\label{lemma:rrf2}\citep{zhou2018stochastic}
If $M_t \geq 2\rho$, then we have
\begin{align}
    -\mineig(\HF(\xb_t + \hb)) \leq M_t\|\hb\|_2 + \big\|\HF(\xb_t) - \Ub_t\big\|_2 + M_t\big|\|\hb\|_2 - \|\hb_t^*\|_2\big|.\notag
\end{align}
\end{lemma}
\begin{proof}[Proof of Lemma \ref{mu_and_h}]
By Lemma \ref{lemma:rrf1}, we have
\begin{align}
    \|\nabla F(\xb_t + \hb)\|_2^{3/2} 
    & \leq \Big[M_t\|\hb\|_2^2 +\big\|\dF(\xb_t) - \vb_t\big\|_2 +  \frac{1}{M_t}\big\|\HF(\xb_t) - \Ub_t\big\|_2^2 + \|\nabla m_t(\hb)\|_2\Big]^{3/2}\notag \\
    & \leq 2\Big[M_t^{3/2}\|\hb\|_2^3 +\big\|\dF(\xb_t) - \vb_t\big\|_2^{3/2} +  M_t^{-3/2}\big\|\HF(\xb_t) - \Ub_t\big\|_2^3 + \|\nabla m_t(\hb)\|_2^{3/2}\Big],\label{eq:mu_and_h:1}
\end{align}
where the second inequality holds due to the fact that for any $a,b,c \geq 0$, we have $(a+b+c)^{3/2} \leq 2(a^{3/2} + b^{3/2} + c^{3/2})$. 
By Lemma \ref{lemma:rrf2}, we have
\begin{align}
    -\rho^{-3/2}\mineig(\HF(\xb_t + \hb))^{3}\notag 
    & \leq \rho^{-3/2}\Big[M_t\|\hb\|_2 + \big\|\HF(\xb_t) - \Ub_t\big\|_2 + M_t\big|\|\hb\|_2 - \|\hb_t^*\|_2\big|\Big]^3\notag \\
    & \leq 9\rho^{-3/2}\Big[M_t^3\|\hb\|_2^3 + \big\|\HF(\xb_t) - \Ub_t\big\|_2^3 + M_t^3\big|\|\hb\|_2 - \|\hb_t^*\|_2\big|^3\Big],\label{eq:mu_and_h:2}
\end{align}
where the second inequality holds due to the fact that for any $a,b,c \geq 0$, we have $(a+b+c)^{3} \leq 9(a^{3} + b^{3} + c^{3})$. 
Thus we have
\begin{align}
    \mu(\xb_t+\hb)& = \max\{\|\nabla F(\xb_t + \hb)\|_2^{3/2}, -\rho^{-3/2}\mineig(\HF(\xb_t + \hb))^{3}\} \notag \\
    & \leq 9 \Big[M_{t}^{3}\rho^{-3/2}\|\hb\|_2^3 + M_{t}^{3/2}\rho^{-3/2} \big\|\dF(\xb_t) - \vb_t\big\|_2^{3/2}+ \rho^{-3/2} \big\|\HF(\xb_t) - \Ub_t\big\|_2^3 \notag \\
 &\qquad+ M_t^{3/2}\rho^{-3/2}\|\nabla m_t(\hb)\|_2^{3/2} + M_t^3\rho^{-3/2}\big|\|\hb\|_2 - \|\hb_t^*\|_2 \big|^3\Big],\notag
\end{align}
where the inequality holds due to \eqref{eq:mu_and_h:1}, \eqref{eq:mu_and_h:2} and the fact that $M_t \geq 4\rho$.
\end{proof}
\subsection{Proof of Lemma \ref{lemma:youngs}}
\begin{proof}[Proof of Lemma \ref{lemma:youngs}]
We have
\begin{align}
    \la\dF(\xb_t) - \vb_t,\hb\ra \leq \big\|\dF(\xb_t) - \vb_t\big\|_2\|\hb\|_2 \leq \frac{\rho}{8} \|\hb\|_2^3 + \frac{6 \|\dF(\xb_t) - \vb_t\|_2^{3/2}}{5\sqrt{\rho}},\notag
\end{align}
where the first inequality holds due to Cauchy–Schwarz inequality, the second inequality holds due to Young's inequality. We also have
\begin{align}
     \big\la\big(\nabla^2F(\xb_t) - \Ub_t\big)\hb,\hb\big\ra &\leq \big\| \nabla^2F(\xb_t) - \Ub_t\big\|_2\|\hb\|_2^2\leq \frac{\rho}{8} \|\hb\|_2^3 + \frac{10}{\rho^2} \big\|\nabla^2F(\xb_t) - \Ub_t\big\|_2^3,\notag
\end{align}
where the first inequality holds due to Cauchy–Schwarz inequality, the second inequality holds due to Young's inequality. 
\end{proof}

\subsection{Proof of Lemma \ref{gradientVariance}}
We need the following lemma:
\begin{lemma}\label{vhoeff_prac}
Conditioned on $\cF_k$, with probability at least $1-\delta$ , we have 
\begin{align}
    &\big\|\df_{\cJ_k}(\xb_k) - \df_{\cJ_k}(\xb_{k-1}) - \dF(\xb_k) + \dF(\xb_{k-1})\big\|_2 \leq 6L\sqrt{\frac{\log(1/\delta)}{B^{(g)}_k}}\|\xb_k - \xb_{k-1}\|_2.\label{hoeff_prac_0}
\end{align}
We also have
\begin{align}
    \|\df_{\cJ_k}(\xb_k) - \dF(\xb_k)\|_2 \leq 6M\sqrt{\frac{\log(1/\delta)}{B^{(g)}_k}}.\label{hoeff_prac_1}
\end{align}
\end{lemma}

\begin{proof}[Proof of Lemma \ref{gradientVariance}]
First, we have $\vb_t - \dF(\xb_t) = \sum_{k=\lfloor t/S^{(g)}\rfloor \cdot S^{(g)}}^t \ub_k$, where 
\begin{align}
    \ub_k &= \df_{\cJ_k}(\xb_k) - \df_{\cJ_k}(\xb_{k-1}) - \dF(\xb_k) + \dF(\xb_{k-1}),&k>\lfloor t/S^{(g)}\rfloor \cdot S^{(g)},\notag \\ \ub_k &= \df_{\cJ_k}(\xb_k) - \dF(\xb_k), &k = \lfloor t/S^{(g)}\rfloor \cdot S^{(g)}\notag
\end{align}
Meanwhile, we have $\EE[\ub_k|\cF_{k-1}] = 0$. Conditioned on $\cF_{k-1}$, for $\mod(k, S^{(g)}) \neq 0$, from Lemma \ref{vhoeff_prac}, we have that with probability at least $1-\delta$ the following inequality holds :
\begin{align}
    \|\ub_k\|_2 \leq 6L\sqrt{\frac{\log(1/\delta)}{B^{(g)}_k}}\|\xb_k - \xb_{k-1}\|_2 \leq \sqrt{\frac{\epsilon^2}{540S^{(g)}\log(1/\delta)}},\label{gradientVar_0}
\end{align}
where the second inequality holds due to \eqref{gradientVariance-1}. For 
$\mod(k, S^{(g)}) = 0$, with probability at least $1-\delta$, we have
\begin{align}
    \|\ub_k\|_2 \leq 6M\sqrt{\frac{\log(1/\delta)}{B^{(g)}_k}} \leq \frac{\epsilon}{\sqrt{540\log(1/\delta)}} ,\label{gradientVar_1}
\end{align}
where the second inequality holds due to \eqref{gradientVariance-2}.
Conditioned on $\cF_{\lfloor t/S^{(g)}\rfloor\cdot S^{(g)}}$, by union bound, with probability at least $1-\delta\cdot(t - \lfloor t/S^{(g)}\rfloor\cdot S^{(g)})$ \eqref{gradientVar_0} or \eqref{gradientVar_1} holds for all $\lfloor t/S^{(g)}\rfloor\cdot S^{(g)} \leq k \leq t$.
Then for given $k$, by vector Azuma-Hoeffding inequality in Lemma \ref{vazuma}, conditioned on$\cF_{k}$, with probability at least $1-\delta$ we have
\begin{align}
    \|\vb_k - \dF(\xb_k)\|_2^2 & = \bigg\|\sum_{k=\lfloor t/S^{(g)}\rfloor \cdot S^{(g)}}^t \ub_k\bigg\|_2^2 \notag \\
    & \leq 9 \log(d/\delta)\Big[(t - \lfloor t/S^{(g)}\rfloor\cdot S^{(g)}) \cdot \frac{\epsilon^2}{540S^{(g)}\log(d/\delta)} + \frac{\epsilon^2}{540\log(1/\delta)}\Big]\notag \\
    & \leq 9 \log(1/\delta)\cdot \frac{\epsilon^2}{270\log(1/\delta)}\notag \\
    & \leq \epsilon^2/30.\label{gradientVar_3}
\end{align}
Finally, by union bound, we have that with probability at least $1-2\delta\cdot(t - \lfloor t/S^{(g)}\rfloor\cdot S^{(g)})$, for all $\lfloor t/S^{(g)}\rfloor\cdot S^{(g)} \leq k \leq t$, we have \eqref{gradientVar_3} holds.
\end{proof}

\subsection{Proof of Lemma \ref{HessianVariance}}
We need the following lemma:
\begin{lemma}\label{hoeff_prac}
Conditioned on $\cF_k$, with probability at least $1-\delta$ , we have the following concentration inequality
\begin{align}
    &\big\|\Hf_{\cI_k}(\xb_k) - \Hf_{\cI_k}(\xb_{k-1}) - \nabla^2 F(\xb_k) + \nabla^2 F(\xb_{k-1})\big\|_2 \leq 6\rho\sqrt{\frac{\log(d/\delta)}{B^{(h)}_k}}\|\xb_k - \xb_{k-1}\|_2.\label{hoeff_prac_0}
\end{align}
We also have
\begin{align}
    \|\Hf_{\cI_k}(\xb_k) - \HF(\xb_k)\|_2 \leq 6L\sqrt{\frac{\log(d/\delta)}{B^{(h)}_k}}.\label{hoeff_prac_1}
\end{align}
\end{lemma}
\begin{proof}[Proof of Lemma \ref{HessianVariance}]
First, we have $\Ub_t - \HF(\xb_t) = \sum_{k=\lfloor t/S^{(h)}\rfloor \cdot S^{(h)}}^t \Vb_k$, where 
\begin{align}
    \Vb_k &= \Hf_{\cI_k}(\xb_k) - \Hf_{\cI_k}(\xb_{k-1}) - \HF(\xb_k) + \HF(\xb_{k-1}),&k>\lfloor t/S^{(h)}\rfloor \cdot S^{(h)},\notag \\ 
    \Vb_k &= \df_{\cI_k}(\xb_k) - \dF(\xb_k), &k = \lfloor t/S^{(h)}\rfloor \cdot S^{(h)}\notag
\end{align}
Meanwhile, we have $\EE[\Vb_k|\sigma(\Vb_{k-1}, ..., \Vb_0)] = 0$. Conditioned on $\cF_{k-1}$, for $\mod(k, S^{(h)}) \neq 0$, from Lemma \ref{hoeff_prac}, we have that with probability at least $1-\delta$, the following inequality holds :
\begin{align}
    \|\Vb_k\|_2 \leq 6\rho\sqrt{\frac{\log(d/\delta)}{B^{(h)}_k}}\|\xb_k - \xb_{k-1}\|_2  \leq  \sqrt{\frac{\rho\epsilon}{360S^{(h)}\log(d/\delta)}},\label{hessianva_0}
\end{align}
where the second inequality holds due to \eqref{gradientVariance-1}. For 
$\mod(k, S^{(h)}) = 0$, with probability at least $1-\delta$, we have
\begin{align}
    \|\Vb_k\|_2 \leq 6L\sqrt{\frac{\log(d/\delta)}{B^{(h)}_k}} \leq \sqrt{\frac{\rho\epsilon}{360\log(d/\delta)}},\label{hessianva_1}
\end{align}
where the second inequality holds due to \eqref{gradientVariance-2}.
Conditioned on $\cF_{\lfloor t/S^{(h)}\rfloor\cdot S^{(h)}}$, by union bound, with probability at least $1-\delta\cdot(t - \lfloor t/S^{(h)}\rfloor\cdot S^{(h)})$ \eqref{hessianva_0} or \eqref{hessianva_1} holds for all $\lfloor t/S^{(h)}\rfloor\cdot S^{(h)} \leq k \leq t$.
Then for given $k$, by Matrix Azuma inequality Lemma \ref{Azuma}, conditioned on$\cF_{k}$, with probability at least $1-\delta$ we have
\begin{align}
    \|\Ub_k - \HF(\xb_k)\|_2^2 
    & = \bigg\|\sum_{k = \lfloor t/S^{(h)}\rfloor\cdot S^{(h)}}^t \Vb_k\bigg\|_2^2\notag \\
   & \leq 9 \log(d/\delta)\Big[(t - \lfloor t/S^{(h)}\rfloor\cdot S^{(h)}) \cdot \frac{\rho\epsilon}{360S^{(h)}\log(d/\delta)} + \frac{\rho\epsilon}{360\log(d/\delta)}\Big]\notag \\
    & \leq 9 \log(d/\delta)\cdot \frac{\rho\epsilon}{180\log(d/\delta)}\notag \\
    & \leq \rho\epsilon/20.\label{hessianva_3}
\end{align}
Finally, by union bound, we have that with probability at least $1-2\delta\cdot(t - \lfloor t/S^{(h)}\rfloor\cdot S^{(h)})$, for all $\lfloor t/S^{(h)}\rfloor\cdot S^{(h)} \leq k \leq t$, we have \eqref{hessianva_3} holds.

\end{proof}

\section{Proofs of Technical Lemmas in Appendix \ref{app:b}}
\subsection{Proof of Lemma \ref{sufficientde}}
We have the following lemma which guarantees the effectiveness of $\namesub$ in Algorithm \ref{alg:sub}.
\begin{lemma}\label{carmon_0}\citep{Carmon2016Gradient}
Let $\Ab \in \RR^{d \times d}$ and $\|\Ab\|_2 \leq \beta$, $\bbb \in \RR^d$, $\tau>0, \zeta>0, \epsilon' \in(0,1), \delta' \in (0,1)$ and $\eta < 1/(8\beta+2\tau\zeta)$. We denote that $g(\hb) = \bbb^\top\hb + \hb^\top\Ab\hb/2 + \tau/6\cdot\|\hb\|_2^3$ and $\sbb = \argmin_{\hb \in \RR^d}g(\hb)$. Then with probability at least $1-\delta'$, if
\begin{align}
    \|\sbb\|_2 \geq \zeta\text{ or }\|\bbb\|_2 \geq \max\{ \sqrt{\beta\tau/2}\zeta^{3/2}, \tau\zeta^2/2\},\label{carmon_0_0}
\end{align}
then $\xb = \namesub(\Ab, \bbb, \tau, \eta, \zeta, \epsilon', \delta')$ satisfies that $g(\xb) \leq -(1-\epsilon')\tau\zeta^3/12$.
\end{lemma}
\begin{proof}[Proof of Lemma \ref{sufficientde}]
We simply set $\Ab = \Ub_t$, $\bbb = \vb_t$, $\tau = M_t$, $\eta = (16L)^{-1}$, $\zeta = \sqrt{\epsilon/\rho}$, $\epsilon' = 0.5$ and $\delta' = \delta$. We have $\|\Ub_t\|_2 \leq L$, then we set $\beta = L$. With the choice of $M_t$ where $M_t = 4\rho$ and the assumption that $\epsilon < 4L^2\rho/M_t^2$, we can check that $\eta < 1/(8\beta + 2\tau\zeta)$. We also have that $\sbb = \hb_t^*$ and \eqref{carmon_0_0} holds. Thus, by Lemma \ref{carmon_0}, we have 
\begin{align}
    m_t(\hb_t) \leq -(1-\epsilon')\tau\zeta^3/12 \leq -M_t\rho^{-3/2}\epsilon^{3/2}/24.\notag
\end{align}
By the choice of $T'$ in $\namesub$, we have  
\begin{align}
    T' = \frac{480}{\eta\tau\zeta\epsilon'}\bigg[6\log\bigg(1+\sqrt{d}/\delta'\bigg) + 32\log\bigg(\frac{12}{\eta\tau\zeta\epsilon'}\bigg)\bigg)\bigg] = \tilde O\bigg(\frac{L}{M_t\sqrt{\epsilon/\rho}}\bigg).\notag
\end{align}
\end{proof}

\subsection{Proof of Lemma \ref{finalitera}}
We have the following lemma which provides the guarantee for the function value in $\namefinal$.
\begin{lemma}\label{carmon_2}\citep{Carmon2016Gradient}
We denote that $g(\hb) = \bbb^\top\hb + \hb^\top\Ab\hb/2 + \tau/6\cdot\|\hb\|_2^3$, $\sbb = \argmin_{\hb \in \RR^d}g(\hb)$, then $g(\sbb) \geq \|\bbb\|_2\|\sbb\|_2/2 - \tau\|\sbb\|_2^3/6$.
\end{lemma}
\begin{proof}[Proof of Lemma \ref{finalitera}]
In $\namefinal$ we are focusing on minimizing $m_{T^*-1}(\hb)$. We have that $\|\vb_t\|_2 < \max\{M_t\epsilon/(2\rho), \sqrt{LM_t/2}(\epsilon/\rho)^{3/4}\} $ and $\|\hb_{T^*-1}^*\|_2 \leq \sqrt{\epsilon/\rho}$ by Lemma \ref{sufficientde}.
We can check that $\eta = (16L)^{-1}$ satisfies that $\eta<(4(L+\tau R))^{-1}$, where $R$ is defined in Lemma \ref{carmon_1}, when $\epsilon < 4L^2\rho/M_t^2$. From Lemma \ref{carmon_1} we also know that $m_{T^*-1}$ is $(L+2M_{T^*-1}R)$-smooth, which satisfies that $1/\eta>2(L+2M_{T^*-1}R)$. Thus, by standard gradient descent analysis, to get a point $\Delta$ where $\|\nabla m_{T^*-1}(\Delta)\|_2 \leq \epsilon$, $\namefinal$ needs to run
\begin{align}
    T'' = O\bigg(\frac{m_{T^*-1}(\Delta_0) - m_{T^*-1}(\hb^*_{T^*-1})}{\eta\epsilon^2}\bigg) = O\bigg(L\frac{m_{T^*-1}(\Delta_0) - m_{T^*-1}(\hb^*_{T^*-1})}{\epsilon^2}\bigg)\label{count_10}
\end{align}
iterations, where we denote by $\Delta_0$ the starting point of $\namefinal$. By directly computing, we have $m_{T^*-1}(\Delta_0) \leq 0$. By Lemma \ref{carmon_2}, we have 
\begin{align*}
    -m_{T^*-1}(\hb^*_{T^*-1}) &\leq M_t\|\hb^*_{T^*-1}\|_2^3/6 - \|\vb_{T^*-1}\|_2\|\hb_{T^*-1}^*\|_2/2 \\
    &\leq M_t\|\hb^*_{T^*-1}\|_2^3/6\\
    &= O\big( \rho(\epsilon/\rho)^{3/2}\big) \\
    &= O(\epsilon^{3/2}/\sqrt{\rho}).\notag 
\end{align*}
Thus, \eqref{count_10} can be further bounded as $T'' = O(L/\sqrt{\rho\epsilon})$. 
\end{proof}

\section{Proofs of Additional Lemmas in Appendix \ref{app_c}}
\subsection{Proof of Lemma \ref{vhoeff_prac}}

\begin{proof}[Proof of Lemma \ref{vhoeff_prac}]
We only need to consider the case where $B^{(g)}_k = |\cJ_k|<n$. For each $i \in \cJ_k$, let
\begin{align}
    \ab_i = \df_i(\xb_k) - \df_i(\xb_{k-1}) - \dF(\xb_k) + \dF(\xb_{k-1}),\notag
\end{align}
then we have $\EE_i\ab_i = 0$, $\ab_i$  i.i.d., and 
\begin{align}
    \|\ab_i\|_2 \leq \|\df_i(\xb_k) - \df_i(\xb_{k-1})\|_2 +\| \dF(\xb_k) - \dF(\xb_{k-1})\|_2 \leq 2L\|\xb_k - \xb_{k-1}\|_2,\notag
\end{align}
where the second inequality holds due to the $L$-smoothness of $f_i$ and $F$. Thus by vector Azuma-Hoeffding inequality in Lemma \ref{vazuma}, we have that with probability at least $1-\delta$, 
\begin{align}
    &\big\|\df_{\cJ_k}(\xb_k) - \df_{\cJ_k}(\xb_{k-1}) - \dF(\xb_k) + \dF(\xb_{k-1})\big\|_2 \notag \\
    & = \frac{1}{B^{(g)}_k}\bigg\|\sum_{i \in \cJ_k} \Big[\df_i(\xb_k) - \df_i(\xb_{k-1}) - \dF(\xb_k) + \dF(\xb_{k-1})\Big]\bigg\|_2\notag \\
    &\leq 6L\sqrt{\frac{\log(d/\delta)}{B^{(g)}_k}}\|\xb_k - \xb_{k-1}\|_2 .\notag
\end{align}
For each $i \in \cJ_k$, let 
\begin{align}
    \bbb_i = \df_i(\xb_k) - \dF(\xb_k),\notag
\end{align}
then we have $\EE_i\bbb_i = 0$ and $\|\bbb_i\|_2 \leq \upp$. Thus by vector Azuma-Hoeffding inequality in Lemma \ref{vazuma}, we have that with probability at least $1-\delta$, 
\begin{align}
    \|\df_{\cJ_k}(\xb_k) - \dF(\xb_k)\|_2 = \frac{1}{B^{(g)}_k}\bigg\| \sum_{i \in \cJ_k}\Big[\df_i(\xb_k) - \dF(\xb_k)\Big]\bigg\|_2 \leq 6\upp\sqrt{\frac{\log(d/\delta)}{B^{(g)}_k}}.\notag
\end{align}
\end{proof}
\subsection{Proof of Lemma \ref{hoeff_prac}}
\begin{proof}[Proof of Lemma \ref{hoeff_prac}]
We only need to consider the case where $B^{(h)}_k = |\cI_k|<n$. For each $i \in \cI_k$, let
\begin{align}
    \Ab_i = \Hf_{i}(\xb_k) - \Hf_{i}(\xb_{k-1}) - \nabla^2 F(\xb_k) + \nabla^2 F(\xb_{k-1}),\notag
\end{align}
then we have $\EE_i \Ab_i = 0, \Ab_i^\top = \Ab_i$, $\Ab_i$ i.i.d. and 
\begin{align}
    \|\Ab_i\|_2 \leq \big\|\Hf_{i}(\xb_k) - \Hf_{i}(\xb_{k-1})\big\|_2 +\big\| \nabla^2 F(\xb_k) - \nabla^2 F(\xb_{k-1})\big\|_2 \leq 2\rho \|\xb_k - \xb_{k-1}\|_2,\notag
\end{align}
where the second inequality holds due to $\rho$-Hessian Lipschitz continuous of $f_i$ and $F$. Then by Matrix Azuma inequality Lemma \ref{Azuma}, we have that with probability at least $1-\delta$,
\begin{align}
    &\big\|\Hf_{\cI_k}(\xb_k) - \Hf_{\cI_k}(\xb_{k-1}) - \nabla^2 F(\xb_k) + \nabla^2 F(\xb_{k-1})\big\|_2 \notag \\
    &= 
    \frac{1}{B^{(h)}_k}\bigg\|\sum_{i \in \cI_k}\Big[\Hf_{i}(\xb_k) - \Hf_{i}(\xb_{k-1}) - \nabla^2 F(\xb_k) + \nabla^2 F(\xb_{k-1})\Big]\bigg\|_2\notag \\
    &\leq 6\rho\sqrt{\frac{\log(d/\delta)}{B^{(h)}_k}}\|\xb_k - \xb_{k-1}\|_2.\notag
\end{align}
For each $i \in \cI_k$, let 
\begin{align}
    \Bb_i = \Hf_{i}(\xb_k) - \HF(\xb_k),\notag
\end{align}
then we have $\EE_i \Bb_i=  0$, $\Bb_i^\top  = \Bb_i$, and $\|\Bb_i\|_2 \leq 2L$. Then by Matrix Azuma inequality in Lemma \ref{Azuma}, we have that with probability at least $1-\delta$,
\begin{align}
    \|\Hf_{\cJ_k}(\xb_k) - \HF(\xb_k)\|_2  = \frac{1}{B^{(h)}_k}\bigg\|\sum_{i \in \cI_k}\Big[\Hf_{i}(\xb_k) - \HF(\xb_k)\Big]\bigg\|_2\leq 6L\sqrt{\frac{\log(d/\delta)}{B^{(h)}_k}},\notag
\end{align}
which completes the proof.
\end{proof}

\section{Auxiliary Lemmas}

We have the following vector Azuma-Hoeffding inequality:
\begin{lemma}\label{vazuma}\citep{pinelis1994optimum}
Consider $\{\vb_k\}$ be a vector-valued martingale difference, where $\EE[\vb_k|\sigma(\vb_1,...,\vb_{k-1})] = 0$ and $\|\vb_k\|_2 \leq A_k$, then we have that with probability at least $1-\delta$,
\begin{align}
    \bigg\|\sum_k \vb_k\bigg\|_2 \leq 3\sqrt{\log(1/\delta)\sum_k A_k^2}.\notag
\end{align}
\end{lemma}
We have the following Matrix Azuma inequality :
\begin{lemma}\label{Azuma}\citep{tropp2012user}
Consider a finite adapted sequence $\{\Xb_k\}$ of self-adjoint matrices in dimension $d$, and a fixed sequence $\{\Ab_k\}$ of self-adjoint matrices that satisfy
\begin{align}
    \EE[\Xb_k|\sigma(\Xb_{k-1},...,\Xb_1)] = {\textbf 0}\text{ and }\Xb_k^2 \preceq \Ab_k^2\text{ almost surely.}\notag
\end{align}
Then we have that with probability at least $1-\delta$,
\begin{align}
    \bigg\| \sum_k \Xb_k \bigg\|_2 \leq 3\sqrt{\log(d/\delta) \sum_k\|\Ab_k\|_2^2}.\notag
\end{align}
\end{lemma}

\section{Additional Algorithms and Functions}\label{appendix: add_apgorithm}

Due to space limit, we include the approximate solvers \citep{Carmon2016Gradient} for the cubic subproblem in this section for the purpose of self-containedness. 

\begin{algorithm*}[h]
\caption{$\namesub(\Ab[\cdot], \bbb, \tau, \eta, \zeta, \epsilon', \delta')$}\label{alg:sub}
\begin{algorithmic}[1]
\STATE $\xb = \text{CauchyPoint}(\Ab[\cdot], \bbb,\tau)$
\IF{CubicFunction$(\Ab[\cdot], \bbb,\tau,\xb) \leq -(1-\epsilon')\tau\zeta^3/12$}
\STATE \textbf{return} $\xb$
\ENDIF
  \STATE Set
  \begin{align}
      T' = \frac{480}{\eta\tau\zeta\epsilon'}\bigg[6\log\bigg(1+\sqrt{d}/\delta'\bigg) + 32\log\bigg(\frac{12}{\eta\tau\zeta\epsilon'}\bigg)\bigg)\bigg]\notag
  \end{align}
  \STATE Draw $\qb$ uniformly from the unit sphere, set $\tilde\bbb = \bbb + \sigma\qb$ where $\sigma = \tau^2\zeta^3\epsilon'/(\beta+\tau\zeta)/576$
  \STATE $\xb = \text{CauchyPoint}(\Ab[\cdot], \bbb, \tau)$
  \FOR{$t=1,\ldots,T-1$}
  \STATE $\xb \leftarrow \xb - \eta\cdot \text{CubicGradient}(\Ab[\cdot],\tilde\bbb, \tau,\xb)$

  \IF {CubicFunction$(\Ab[\cdot], \tilde\bbb, \tau,\xb) \leq -(1-\epsilon')\tau\zeta^3/12$}
  \STATE \textbf{return}  $\xb$
  \ENDIF

  \ENDFOR
  \STATE \textbf{return} $\xb$

\end{algorithmic}
\end{algorithm*}

\begin{algorithm*}[h]
\caption{$\namefinal(\Ab[\cdot], \bbb, \tau, \eta, \epsilon_g)$}\label{alg:finalsub}
\begin{algorithmic}[1]
  \STATE $\Delta \leftarrow $CauchyPoint$(\Ab[\cdot], \bbb, \tau)$
  \WHILE{$\|\text{Gradient}(\Ab[\cdot], \bbb, \tau, \Delta)\|_2> \epsilon_g$}
  \STATE $\Delta \leftarrow \Delta - \eta\cdot \text{Gradient}(\Ab[\cdot], \bbb, \tau, \Delta)$
  \ENDWHILE
  \STATE \textbf{return} $\Delta$
\end{algorithmic}
\end{algorithm*}

\begin{algorithm*}[t!]
\begin{algorithmic}[1]
  \STATE \textbf{Function:} CauchyPoint$(\Ab[\cdot], \bbb,\tau )$
  \STATE \textbf{return} $-R_c\bbb/\|\bbb\|_2$, where
  \begin{align}
      R_c = \frac{-\bbb^\top\Ab[\bbb]}{\tau\|\bbb\|_2^2} + \sqrt{\bigg(\frac{-\bbb^\top\Ab[\bbb]}{\tau\|\bbb\|_2^2}\bigg)^2 + \frac{2\|\bbb\|_2}{\tau}}\notag
  \end{align}
  
  \hrulefill	

  \STATE \textbf{Function:} CubicFunction$(\Ab[\cdot], \bbb,\tau ,\xb)$
  \STATE \textbf{return} $\bbb^\top\xb + \xb^\top\Ab[\xb]/2 + \tau\|\xb\|_2^3/6$
  
  \hrulefill	

  \STATE \textbf{Function:} CubicGradient$(\Ab[\cdot], \bbb,\tau ,\xb)$
  \STATE \textbf{return} $\bbb^\top + \Ab[\xb] + \tau\|\xb\|_2\xb/2$
\end{algorithmic}
\end{algorithm*}

\newpage
\bibliographystyle{ims}
\bibliography{reference}

\end{document}